\documentclass[12pt,reqno]{amsart}
\usepackage{amssymb}
\usepackage{amscd}
\usepackage{hyperref}

\newcommand{\func}{\operatorname}
\newtheorem{theorem}{Theorem}[section]
\newtheorem{corollary}[theorem]{Corollary}
\newtheorem{lemma}[theorem]{Lemma}
\newtheorem{proposition}[theorem]{Proposition}

\newtheorem{definition}[theorem]{Definition}
\newtheorem{remark}[theorem]{Remark}

\newtheorem{example}[theorem]{Example}

\numberwithin{equation}{section}

\begin{document}
\title{Perturbations of basic Dirac operators on Riemannian foliations}
\author{Igor Prokhorenkov}
\author{Ken Richardson}
\address{Department of Mathematics\\
Texas Christian University\\
Box 298900 \\
Fort Worth, Texas 76129}
\email{i.prokhorenkov@tcu.edu\\
k.richardson@tcu.edu}
\date{May, 2013}

\begin{abstract}
Using the method of Witten deformation, we express the basic index of a
transversal Dirac operator over a Riemannian foliation as the sum of
integers associated to the critical leaf closures of a given foliated bundle
map.
\end{abstract}

\maketitle


\section{Introduction}

\vspace{0in}It is well-known that the index of the de Rham operator 
\begin{equation*}
D=d+d^{\ast }:\Omega ^{\mathrm{even}}\left( M\right) \rightarrow \Omega ^{%
\mathrm{odd}}\left( M\right)
\end{equation*}%
defined by%
\begin{eqnarray*}
\mathrm{ind}\left( D\right) &=&\dim \ker \left( \left. D\right\vert _{\Omega
^{\mathrm{even}}}\right) -\dim \ker \left( \left. D\right\vert _{\Omega ^{%
\mathrm{odd}}}\right) \\
&=&\dim \ker \left( \left. D^{2}\right\vert _{\Omega ^{\mathrm{even}%
}}\right) -\dim \ker \left( \left. D^{2}\right\vert _{\Omega ^{\mathrm{odd}%
}}\right)
\end{eqnarray*}%
is the Euler characteristic of the Riemannian manifold $M$ (compact, no
boundary). In \cite{Wit1}, Witten replaced $d$ with the deformed
differential $d_{s}=e^{-sf}d~e^{sf}$, where $s>0$ and $f$ is a smooth
real-valued function. This leads to a one-parameter family of deformed Dirac
operators 
\begin{equation*}
D_{s}=d_{s}+d_{s}^{\ast }=D+sZ,
\end{equation*}%
where $Z=df\wedge +\left( df\wedge \right) ^{\ast }$. This family of
Fredholm operators has the same index for all $s$, since the index is
invariant under homotopy. Witten's idea was that each eigenvalue of $%
D_{s}^{2}$ has an asymptotic expansion as $s\rightarrow \infty $ with the
leading term computable from the local data at the critical set of $f$
(where $df=0$). In particular, if $f$ is a Morse function, one can show that 
\begin{equation*}
\mathrm{ind}\left( D\right) =\sum \left( -1\right) ^{p}m_{p},
\end{equation*}%
where $m_{p}$ is the number of critical points of $f$ of Morse index $p$.

In \cite{PrRiPerturb}, the authors expanded the method of Witten deformation
and provided the formula for the index in all cases when $D$ is a Dirac-type
operator and $Z$ is an admissible bundle map satisfying certain
nondegeneracy conditions (similar to the Morse conditions).

The purpose of this paper is to find an expression for the basic index of a
transversal Dirac operator over a Riemannian foliation in terms of local
quantities associated to the singular set of a foliated bundle map
satisfying admissibility and nondegeneracy conditions (see Sections \ref%
{admissiblePertSection} and \ref{properPerturbSection}). We use the method
of Witten deformation to achieve localization (see Corollary \ref%
{LocalizationCorollary}). Our main result (Theorem \ref{Masterpiece}) is the
formula%
\begin{equation*}
\mathrm{ind}_{b}\left( D_{b}\right) =\sum_{\ell }\left( \dim \left[
\bigcap_{j}\left( \bigoplus_{\lambda <0}E_{\lambda }\left( L_{j}^{+}\left( 
\overline{x}\right) \right) \right) ^{H_{\overline{x}}}\right] -\dim \left[
\bigcap_{j}\left( \bigoplus_{\lambda <0}E_{\lambda }\left( L_{j}^{-}\left( 
\overline{x}\right) \right) \right) ^{H_{\overline{x}}}\right] \right) .
\end{equation*}%
Here, $D_{b}$ is a basic Dirac operator, $\ell $ is a critical leaf closure
for a basic bundle map $Z$, and $\overline{x}$ is an arbitrary point of $%
\ell $. The linear maps $L_{j}^{\pm }\left( \overline{x}\right) $ are
obtained from the linearization of the Clifford form of $Z$ at $\overline{x}$
(see Section \ref{SpecialFormSection}), the space $E_{\lambda }$ is the
eigenspace associated to the eigenvalue $\lambda $, and $H_{\overline{x}}$
is the infinitesimal holonomy group associated to $\overline{x}$. The Hopf
index theorem for Riemannian foliations proved in \cite{B-P-R} can be easily
derived from this formula.

We now briefly explain the setup for the formula; precise definitions are in
Section \ref{preliminaryperturb}. The reader may consult the introduction in 
\cite{HabRi1} and Section 3 of \cite{PrRi} for more complete expositions
concerning basic Dirac operators, and more information on Riemannian
foliations is contained in \cite{To} and \cite{Mo}. A Dirac operator has the
form $D=\sum c\left( e_{j}\right) \nabla _{{e_{j}}}^{E}\,,$ where $\nabla
^{E}$ is a Clifford connection on a Hermitian vector bundle $E$ over $M$, $%
\left\{ e_{j}\right\} $ is a local orthonormal frame of $TM$, and $c$
denotes Clifford multiplication. Suppose now that $M$ has the additional
structure of a {\ }Riemannian foliation $\mathcal{F}$, i.e. a layering of $M$
by immersed submanifolds (leaves), and a transverse Riemannian metric that
is invariant along the leaves. A simple example of this structure is that of
the orbits of a compact Lie group action, where all the orbits have the same
dimension, and where the invariant transverse metric can be obtained as the
average of an arbitrary transverse metric along the orbits (leaves).

On a Riemannian foliation, there exist natural operators called transversal{%
\textbf{\ }}Dirac operators. Consider a Hermitian vector bundle $E$ over $M$
that is a module over the complexified Clifford algebra of the normal bundle 
$N\mathcal{F}$ to the foliation. The formula for the transversal Dirac
operator is the same as that of the ordinary Dirac operator, but the sum is
merely over a local orthonormal frame of the normal bundle $N\mathcal{F}%
\subset TM$. This operator restricts to an operator on $\Gamma _{b}\left(
M,E,\mathcal{F}\right) $, the space of basic sections of $E$. A section $s$
of $E$ is called basic\emph{\ }if it is invariant under parallel translation
along the leaves, i.e. if $\nabla _{X}^{E}s=0$ for all $X\in \Gamma (T%
\mathcal{F})$. The leaf space $M/\mathcal{F}$ may be quite singular; these
basic sections provide a type of desingularization of the leaf space. On a
Riemannian foliation, a transverse Dirac operator maps the space $\Gamma
_{b}\left( M,E,\mathcal{F}\right) $ to itself, but it is not necessarily
symmetric with respect to the $L^{2}$ inner product. A \ basic Dirac
operator $D_{b}$ is the symmetric operator defined as 
\begin{equation*}
D_{b}=\sum c\left( e_{j}\right) \nabla _{{e_{j}}}^{E}-\frac{1}{2}c\left(
\kappa _{b}\right) ,
\end{equation*}%
where the sum is over an orthonormal frame of $N\mathcal{F}$ and $\kappa
_{b} $ is the basic component of the mean curvature one-form.

The basic index of a basic Dirac operator $D_{b}$ is 
\begin{equation*}
\text{ind}_{b}(D_{b}^{+})=\dim \ker D_{b}^{+}-\dim \ker (D_{b}^{-})\,,
\end{equation*}%
where one restricts to the subspace of basic sections of the graded
Hermitian bundle $E=E^{+}\oplus E^{-}$. It turns out that this basic index
is a Fredholm index, implying that the dimensions are finite and that the
index is stable under perturbations. Using Molino theory (see Section \ref%
{MolinoSection}), one may show that this index is equivalent to the
equivariant index of a certain Dirac operator on an $O(n)$-manifold.

It is an important problem to express the basic index in geometric terms. We
refer to \cite[Problem 2.8.9]{EK}, where this problem was first stated
explicitly. One example of an index formula is the Gauss-Bonnett theorem for
Riemannian foliations, proved in \cite{BKR3}. Also, in \cite{BKR3}, J. Br%
\"{u}ning, F. W. Kamber, and Richardson use the equivariant index theorem of 
\cite{BKR2} to give a geometric formula for the basic index of a basic Dirac
operator in terms of Atiyah-Singer type integrands and eta invariants. In 
\cite{GLott}, A. Gorokhovsky and J. Lott proved a different formula for the
basic index of a basic Dirac operator, in the case where all the
infinitesimal holonomy groups of the foliation are connected tori and when
Molino's commuting sheaf is abelian and has trivial holonomy. In contrast,
in our formula the index is a sum of integers, each of which is computed at
a single leaf closure.

A related natural question is whether the basic index of a transversally
elliptic operator on a Riemannian foliation has similar properties to those
of an elliptic operator. Until now, it was not known whether this index is
always zero for transversally elliptic differential operators on a
Riemannian foliation of odd codimension. Previously, this seemingly
elementary fact had only been proved for the basic Euler characteristic (in 
\cite{HabRi2}), and the general case is now proved in Corollary \ref%
{OddCodimImpliesIndexZero}.

In Section \ref{preliminaryperturb}, we review properties of foliated vector
bundles and Dirac operators over Riemannian foliations. In Sections \ref%
{admissiblePertSection} and \ref{properPerturbSection}, we establish
conditions on $Z$ so that the index computation localizes to small
neighborhoods of critical leaf closures. The necessary condition for
localization is proved in Theorem \ref{ZConditionTheorem}. We show in
Section \ref{MolinoSection} how to reduce basic differential operator
calculations on a tubular neighborhood of a leaf closure to equivariant
differential operator calculations in Euclidean space. In Section \ref%
{ShubinLocalization}, we define the model operator and prove the
localization theorem. We show in Section \ref{SpecialFormSection} that $Z$
can be deformed to a special form called Clifford form, which is used in our
formula for $\mathrm{ind}_{b}\left( D_{b}\right) $. We prove the main result
(Theorem \ref{Masterpiece}) in Section \ref{MainTheoremSection}. Finally, in
Section \ref{ExamplesSection}, we apply Theorem \ref{Masterpiece} to compute
the basic Euler characteristic and basic signature of some foliations. In
particular, Example \ref{signatureExample} shows that localization is
possible even when $D_{b}Z+ZD_{b}$ is a first order differential operator.

\section{Perturbing basic Dirac operators\label{perturbDiracOpsSection}}

\subsection{Preliminaries and Notational Conventions\label%
{preliminaryperturb}}

\vspace{0in}Let $(M,\mathcal{F})$ be a closed, connected, smooth manifold $M$
equipped with a smooth foliation $\mathcal{F}$ of codimension $q$. We assume
that $(M,\mathcal{F})$ has additional geometric structure. That is, let $%
Q=TM\diagup T\mathcal{F}\rightarrow M$ be the normal bundle of $\mathcal{F}$%
, and let $g_{Q}$ be a metric on $Q$. We assume that $g_{Q}$ is
holonomy-invariant, meaning that its Lie derivative in directions tangent to 
$\mathcal{F}$ is zero. A foliation along with such a $g_{Q}$ is called a 
\textbf{Riemannian foliation}. When $\left( M,\mathcal{F},g_{Q}\right) $ is
a Riemannian foliation, one may always choose a metric $g$ on $M$ such that
the restriction of $g$ to $\left( T\mathcal{F}\right) ^{\bot }$ agrees with $%
g_{Q}$; such metrics are called \textbf{bundle-like metrics}. Given a
bundle-like metric, the leaves of $\mathcal{F}$ are locally equidistant.
Throughout this paper, we assume that we have chosen a bundle-like metric $g$
on $M$ compatible with $g_{Q}$. See \cite{Rein}, \cite{Mo}, and \cite{To}
for standard facts about Riemannian foliations.

We now recall the definitions (see \cite{KT2} and \cite{Mo}) of foliated
bundle and basic connection. Let $G$ be a compact Lie group. We say that a
principal $G$--bundle $P\rightarrow \left( M,\mathcal{F}\right) $ is a 
\textbf{foliated principal bundle} if it is equipped with a foliation $%
\mathcal{F}_{P}$ (the \textbf{lifted foliation}) such that the distribution $%
T\mathcal{F}_{P}$ is invariant under the right action of $G$, is transversal
to the tangent space to the fiber, and projects onto $T\mathcal{F}$. A
connection $\omega $ on $P$ is called \textbf{adapted to }$\mathcal{F}_{P}$
if the associated horizontal distribution contains $T\mathcal{F}_{P}$. An
adapted connection $\omega $ is called a \textbf{basic connection} if it is
basic as a $\mathfrak{g}$-valued form on $\left( P,\mathcal{F}_{P}\right) $;
that is, $i_{X}\omega =0$ and $i_{X}d\omega =0$ for every $X\in T\mathcal{F}$%
, where $i_{X}$ denotes the interior product with $X$. Note that in \cite%
{KT2} the authors showed that basic connections always exist on a foliated
principal bundle over a Riemannian foliation.

Similarly, a complex vector bundle $E\rightarrow \left( M,\mathcal{F}\right) 
$ of rank $k$ is \textbf{foliated} if $E$ is associated to a foliated
principal bundle $P\rightarrow \left( M,\mathcal{F}\right) $ via a
representation $\rho $ from $G$ to $U\left( k\right) $. Let $\Omega \left(
M,E\right) $ denote the space of forms on $M$ with coefficients in $E$. If a
connection form $\omega $ on $P$ is adapted, then we say that an associated
covariant derivative operator $\nabla ^{E}$ on $\Omega \left( M,E\right) $
is \emph{\ }\textbf{adapted} to the foliated bundle. We say that $\nabla
^{E} $ is a \textbf{basic} \textbf{connection} on $E$ if in addition the
associated curvature operator $\left( \nabla ^{E}\right) ^{2}$ satisfies $%
i_{X}\left( \nabla ^{E}\right) ^{2}=0$ for every $X\in T\mathcal{F}$. Note
that $\nabla ^{E}$ is basic if $\omega $ is basic.

For $m\in M$, let $\mathbb{C}\mathrm{l}(Q_{m})$ be the complex Clifford
algebra associated to the vector space $Q$ and quadratic form $g_{Q}$. We
say that $(M,\mathcal{F},g_{Q})$ is \textbf{transversally spin}$^{c}$ if $Q$
is spin$^{c}$. That is, there exists a complex spinor $\mathbb{C}\mathrm{l}%
(Q)$-bundle $\mathbb{S}$ over $M$ such that, for all $m\in M$, the action of 
$\mathbb{C}\mathrm{l}(Q_{m})$ on $\mathbb{S}_{m}$ is an irreducible
representation of $\mathbb{C}\mathrm{l}\left( Q_{m}\right) $. A simple
example shows that this condition is necessary even for spin$^{c}$
manifolds; if $X$ is any manifold that is not spin$^{c}$, then the product
foliation $X\times X$ would not be transversally spin$^{c}$ even though $%
X\times X$ always admits a complex structure and thus is spin$^{c}$.

Let in addition $E\rightarrow M$ be a Hermitian $\mathbb{C}\mathrm{l}(Q)$
Clifford bundle. We assume that the basic connection $\nabla ^{E}$ is
compatible with these additional structures.

Let $\Gamma \left( M,E\right) $ denote the space of smooth sections of $E$
over $M$. Let $\Gamma _{b}\left( M,E,\mathcal{F}\right) $ denote the \textbf{%
space of basic sections} of $E$, meaning that%
\begin{equation*}
\Gamma _{b}\left( M,E,\mathcal{F}\right) =\left\{ u\in \Gamma \left(
M,E\right) :\nabla _{V}^{E}u=0\text{ for all }V\in \Gamma \left( T\mathcal{F}%
\right) \right\} .
\end{equation*}%
Similarly, a bundle endomorphism $A$ is basic if $\nabla _{X}^{\mathrm{End}%
\left( E\right) }A=0$ for all $X\in \Gamma (T\mathcal{F})$.

Now we define the transversal Dirac operator, which is similar to the
standard Dirac operator but is elliptic only when restricted to directions
orthogonal to the tangent bundle $T\mathcal{F}$ of the foliation. Formally,
the \textbf{transversal Dirac operator }$D_{\mathrm{tr}}$ is the composition
of the maps 
\begin{equation*}
\Gamma \left( E\right) \overset{\left( \nabla ^{E}\right) ^{\mathrm{\mathrm{%
tr}}}}{\longrightarrow }\Gamma \left( Q^{\ast }\otimes E\right) \overset{%
\cong }{\longrightarrow }\Gamma \left( Q\otimes E\right) \overset{c}{%
\longrightarrow }\Gamma \left( E\right) ,
\end{equation*}%
where the last map denotes Clifford multiplication, and the operator $\left(
\nabla ^{E}\right) ^{\mathrm{tr}}$ is the projection of $\nabla ^{E}$ to $%
\Gamma \left( Q^{\ast }\otimes E\right) $. Clifford multiplication by an
element $v\in Q_{x}$ on the fiber $E_{x}$ is denoted by $c\left( v\right) $.
Clifford multiplication by cotangent vectors in $Q^{\ast }$ will use the
same notation: $c\left( \alpha \right) :=c\left( \alpha ^{\#}\right) $,
where $Q_{x}^{\ast }\overset{\#}{\rightarrow }Q_{x}$ is the metric
isomorphism. The transversal Dirac operator fixes the basic sections but is
not symmetric on this subspace. By modifying $D_{\mathrm{tr}}$ by a bundle
map, we obtain a symmetric and essentially self-adjoint operator $D_{b}$ on $%
\Gamma _{b}(E)$. To define $D_{b}$, first let $H=\sum_{i=1}^{p}\pi (\nabla
_{f_{i}}^{TM}f_{i})$ be the mean curvature of the foliation, where $\pi
:TM\rightarrow Q$ denotes the projection and $\{f_{i}\}_{i=1,\cdots ,p}$ is
a local orthonormal frame of $T\mathcal{F}$. Let $\kappa =H^{\flat }$ be the
corresponding $1$-form, so that $H=\kappa ^{\sharp }$. Let $\kappa
_{b}:=P_{b}\kappa $ be the $L^{2}$-orthogonal projection of $\kappa $ onto
the space of basic forms (see \cite{Al}, \cite{PaRi}). We now define the 
\textbf{basic Dirac operator} by 
\begin{equation*}
D_{b}u:=\frac{1}{2}(D_{\mathrm{tr}}+D_{\mathrm{tr}}^{\ast
})u=\sum_{i=1}^{q}c\left( e_{i}\right) \nabla _{e_{i}}^{E}u-\frac{1}{2}%
c\left( \kappa _{b}\right) u,
\end{equation*}%
where $\{e_{i}\}_{i=1,\cdots ,q}$ is a local orthonormal frame of $Q$. A
direct computation shows that $D_{b}$ preserves the basic sections. The
operator $D_{b}:\Gamma _{b}(E)\rightarrow \Gamma _{b}(E)$ is transversally
elliptic, has discrete spectrum, and is Fredholm (\cite{EK}). The spectrum
of $D_{b}$ was shown to depend only on $g_{Q}$ and not on $g$ in \cite%
{HabRi1}, in spite of the fact that the mean curvature and $L^{2}$ inner
product certainly depend on the choice of $g$.

A grading on $E$ is induced by the action of the chirality operator $\gamma $
(in the transverse direction). Recall that if $e_{1},...,e_{q}$ is an
oriented orthonormal basis of $Q_{x}$, then the chirality operator is
multiplication by 
\begin{equation*}
\gamma =i^{k}c(e_{1})...c(e_{q})\in \mathrm{End}\left( E_{x}\right) \mathbf{,%
}
\end{equation*}%
where $k=q/2$ if $q$ is even and $k=\left( q+1\right) /2$ if $q$ is odd. The 
$\pm 1$ eigenspaces of $\gamma $ determine a grading of $E=E^{+}\oplus E^{-}$%
. This grading is called the \textbf{natural grading}. The other possible
gradings are classified in \cite[pp. 319ff]{PrRiPerturb}. Observe that the
chirality operator is a basic bundle map, because for $X\in \Gamma (T%
\mathcal{F})$, $u\in \Gamma \left( E\right) $, and a local framing $\left(
e_{1},...,e_{q}\right) $ of $Q$, 
\begin{equation*}
\nabla _{X}^{\mathrm{End}\left( E\right) }\gamma
=i^{k}\sum_{j=1}^{q}c(e_{1})...c(\nabla _{X}^{E}e_{j})...c(e_{q})=0,
\end{equation*}%
because $\left( M,\mathcal{F},g_{Q}\right) $ is Riemannian and we may choose
the local framing so that each $e_{j}$ is a local basic section. Let $%
D_{b}^{\pm }:\Gamma _{b}\left( M,E^{\pm },\mathcal{F}\right) \rightarrow
\Gamma _{b}\left( M,E^{\mp },\mathcal{F}\right) $ denote the restrictions of 
$D_{b}$ to smooth even and odd sections. The operator $D_{b}^{-}$ is the
adjoint of $D_{b}^{+}$ with respect to the $L^{2}$-metric on the space of
basic sections $\Gamma _{b}\left( M,E,\mathcal{F}\right) $\textbf{\ }defined
by $g$\ and the Hermitian metric on $E$.

\subsection{Admissible basic perturbations\label{admissiblePertSection}}

We wish to perturb the basic Dirac operator by a basic bundle map. Let $%
Z^{+}:\Gamma \left( M,E^{+}\right) \rightarrow \Gamma \left( M,E^{-}\right) $
be a smooth basic bundle map, and we let $Z^{-}$ denote the adjoint of $%
Z^{+} $. The operator $Z$ on $\Gamma \left( M,E\right) $, defined by $%
Z\left( v^{+}+v^{-}\right) =Z^{-}v^{-}+Z^{+}v^{+}$ for any $v^{+}\in
E_{x}^{+}$ and $v^{-}\in E_{x}^{-}$ , is self-adjoint. Let $D_{s}$ denote
the perturbed basic Dirac operator 
\begin{equation}
D_{s}=\left( D_{b}+sZ\right) :\Gamma _{b}\left( M,E,\mathcal{F}\right)
\rightarrow \Gamma _{b}\left( M,E,\mathcal{F}\right) ,
\label{perturbed operator}
\end{equation}%
and define the operators $D_{s}^{\pm }$ by restricting in the obvious ways.

\vspace{0in}The basic index $\mathrm{ind}_{b}\left( D_{b}\right) $ depends
only on the homotopy type of the principal transverse symbol of $D_{b}$ and
satisfies 
\begin{equation*}
\mathrm{ind}_{b}\left( D_{b}\right) =\dim \ker \left( \left. \left(
D_{s}\right) ^{2}\right\vert _{\Gamma _{b}\left( M,E^{+},\mathcal{F}\right)
}\right) -\dim \ker \left( \left. \left( D_{s}\right) ^{2}\right\vert
_{\Gamma _{b}\left( M,E^{-},\mathcal{F}\right) }\right) .
\end{equation*}%
Thus, we need to study the operator 
\begin{equation*}
\left( D_{s}\right) ^{2}=D_{b}^{2}+s\left( ZD_{b}+D_{b}Z\right) +s^{2}Z^{2}.
\end{equation*}%
As will be shown later, the leading order behavior of the eigenvalues of
this operator as $s\rightarrow \infty $ is determined by combinatorial data
at the singular set of the operator $Z$. This \textquotedblleft
localization\textquotedblright\ allows one to compute $\mathrm{ind}%
_{b}\left( D_{b}\right) $ in terms of that data. A sufficient condition for
localization techniques to work is the requirement that the operator $%
ZD_{b}+D_{b}Z$ restricts to a bounded operator on the space of smooth basic
sections $\Gamma _{b}\left( M,E,\mathcal{F}\right) $. We need the following
lemmas.

A foliation $\mathcal{F}$ of codimension $q$ is called \textbf{transversally
parallelizable} if there exists a global basis of $Q$ consisting of basic
vector fields. For any Riemannian foliation $\mathcal{F}$, the lifted
foliation $\widehat{\mathcal{F}}$ on the orthonormal transverse frame bundle 
$\widehat{M}$ is always transversally parallelizable (see \cite{Mo}).

\begin{lemma}
If $\mathcal{F}$ is transversally parallelizable, then the space $\Gamma
_{b}\left( M,E,\mathcal{F}\right) $ is a finitely-generated module over the
space $C_{b}^{\infty }\left( M,\mathcal{F}\right) $ of basic functions.
\end{lemma}

\begin{proof}
If $\mathcal{F}$ is transversally parallelizable and $\left( M,\mathcal{F}%
\right) $ is Riemannian, then the leaf closures are the fibers of a
Riemannian submersion for any choice of bundle-like metric. There is a set $%
\left\{ s_{1},...,s_{k}\right\} $ of basic sections of $E$ such that for
every $x\in M$, $\left\{ s_{1}\left( x\right) ,...,s_{k}\left( x\right)
\right\} $ is a basis of $E_{x}$. Then every basic section can be written as 
$\sum f_{j}\left( x\right) s_{j}\left( x\right) $ for some functions $%
f_{j}\left( x\right) $. For any $X$ tangent to $\mathcal{F}$, $\nabla
_{X}^{E}\left( \sum f_{j}\left( x\right) s_{j}\left( x\right) \right) =0$
implies $\sum \left( Xf_{j}\right) \left( x\right) s_{j}\left( x\right) =0$,
so that $Xf_{j}$ for every $j$ and every $X\in T\mathcal{F}$. Thus, the
functions $f_{j}$ must be basic.
\end{proof}

\begin{lemma}
If $\mathcal{F}$ is transversally parallelizable and if $V$ is a basic
vector field that is tangent to every leaf closure of $\left( M,\mathcal{F}%
\right) $, then $\nabla _{V}^{E}:\Gamma _{b}\left( M,E,\mathcal{F}\right)
\rightarrow \Gamma _{b}\left( M,E,\mathcal{F}\right) $ is a bounded $%
C_{b}^{\infty }\left( M,\mathcal{F}\right) $-linear operator.
\end{lemma}

\begin{proof}
As in the last proof, any section can be written as $\sum f_{j}\left(
x\right) s_{j}\left( x\right) $, and for $V$ basic and tangent to the leaf
closures, we have $Vf_{j}=0$, and $\nabla _{V}^{E}\left( \sum f_{j}\left(
x\right) s_{j}\left( x\right) \right) =\sum f_{j}\left( x\right) \nabla
_{V}^{E}s_{j}\left( x\right) .$ Since $\nabla _{V}^{E}s_{j}$ must be basic,
it is a linear combination $\left( \nabla _{V}^{E}s_{j}\right) \left(
x\right) =\sum a_{jk}\left( x\right) s_{k}\left( x\right) $. Thus, $\nabla
_{V}^{E}$ is bounded, since $a_{jk}$ is bounded on the compact $M$.
\end{proof}

\begin{lemma}
\label{BoundedOpBasicVFieldLemma}If $\left( M,\mathcal{F}\right) $ is any
Riemannian foliation and if $V$ is a basic vector field that is tangent to
every leaf closure of $\left( M,\mathcal{F}\right) $, then $\nabla
_{V}^{E}:\Gamma _{b}\left( M,E,\mathcal{F}\right) \rightarrow \Gamma
_{b}\left( M,E,\mathcal{F}\right) $ is a bounded $C_{b}^{\infty }\left( M,%
\mathcal{F}\right) $-linear operator.
\end{lemma}

\begin{proof}
Let $V$ be a basic vector field that is tangent to every leaf closure of $%
\mathcal{F}$, and let $\widehat{V}$ be its horizontal lift in the
orthonormal transverse frame bundle $\widehat{M}$. Observe that horizontal
lifts of basic vector fields are basic for the lifted foliation, so that $%
\widehat{V}$ is a basic vector field. Since the leaf closures of $\widehat{%
\mathcal{F}}$ are principle bundles over the leaf closures of $\mathcal{F}$, 
$\widehat{V}$ must be tangent to the leaf closures of $\widehat{\mathcal{F}}$%
. Since $\left( \widehat{M},\widehat{\mathcal{F}}\right) $ is transversally
parallelizable, $\widehat{\nabla }_{\widehat{V}}$ is bounded as a linear
operator on the space of all basic sections $\Gamma _{b}\left( \widehat{M}%
,\pi ^{\ast }E,\widehat{\mathcal{F}}\right) $. Since for every section $s\in
\Gamma _{b}\left( M,E,\mathcal{F}\right) $, $\pi ^{\ast }s\in \Gamma
_{b}\left( \widehat{M},\pi ^{\ast }E,\widehat{\mathcal{F}}\right) $, $%
\widehat{\nabla }_{\widehat{V}}$ is bounded as an operator on $\pi ^{\ast
}\Gamma _{b}\left( M,E,\mathcal{F}\right) $. Finally, $\nabla _{V}^{E}s=%
\widehat{\nabla }_{\widehat{V}}\pi ^{\ast }s$ for all $s\in \Gamma
_{b}\left( M,E,\mathcal{F}\right) $, and thus $\nabla _{V}^{E}:\Gamma
_{b}\left( M,E,\mathcal{F}\right) \rightarrow \Gamma _{b}\left( M,E,\mathcal{%
F}\right) $ is a bounded operator. Further, $\nabla _{V}^{E}$ is $%
C_{b}^{\infty }\left( M,\mathcal{F}\right) $-linear for the same reasons as
in the last lemma.
\end{proof}

\begin{lemma}
\label{localAdaptedFrameLemma}If $\left( M,\mathcal{F}\right) $ is any
Riemannian foliation and if $x\in M$, then there exists a local orthonormal
frame $\left\{ e_{1},...,e_{\overline{q}},e_{\overline{q}+1},...,e_{q}\right%
\} $ of $N\mathcal{F}$ near $x$ consisting of basic vector fields such that $%
e_{1},...,e_{\overline{q}}$ is a local transverse orthonormal frame for the
leaf closure containing $x$, and in a neighborhood of $x$ the fields $%
e_{1},...,e_{\overline{q}}$ remain tangent to the leaf closures of $\mathcal{%
F}$.
\end{lemma}

\begin{proof}
From \cite{Mo}, there exists a local orthonormal frame of $N\mathcal{F}$
near $x$ consisting of basic vector fields. It remains to show that the
frame may be chosen to be adapted to the leaf closures near $x$. Let $B$ be
a transversal submanifold to the leaf through $x$. By \cite[Ch. 1, 5, App. D]%
{Mo}, the orbits of the closure of the holonomy pseudogroup acting on $B$
are exactly the intersections of the leaf closures with $B$, and there exist
a smooth family of local isometries that generate the leaf closures near $x$%
. This family is at least $\overline{q}$-dimensional, where $\overline{q}$
is the dimension of the leaf closure through $x$. Choose $\overline{q}$
one-parameter families of local isometries in the pseudogroup such that
their initial velocities generate linearly independent vectors at $x$. By
orthogonalization, we may alter these families so that these velocity
vectors $e_{1},...,e_{\overline{q}}$ at $x$ are orthonormal. The velocity
vectors of the families of isometries naturally extend to a neighborhood of $%
x$, and by construction they are basic fields that are tangent to all leaf
closures near $x$ and are linearly independent. They are automatically
orthonormal on the portion of the leaf closure containing $x$ inside $B$,
and we may orthogonalize over $B$ in a neighborhood of this leaf closure so
that $e_{1},...,e_{\overline{q}}$ is a local orthonormal set in a
neighborhood of $x$, consisting of basic fields tangent to the leaf closures
of $\mathcal{F}$. We may then extend this local frame to a local orthonormal
frame of $Q$ near $x$ in $B$, and then we may extend the frame in a
neighborhood of the transversal so that the fields are basic. Thus, the new
frame $\left\{ e_{1},...,e_{\overline{q}},e_{\overline{q}+1},...,e_{q}\right%
\} $ has the desired properties.
\end{proof}

The following theorem gives a sufficient condition for the operator $%
ZD_{b}+D_{b}Z$ to be bounded on $\Gamma _{b}\left( M,E,\mathcal{F}\right) $.
Let $\sigma _{D_{b}}\left( x,\xi \right) $ denote the principal symbol of $%
D_{b}$ at the covector $\xi \in T_{x}^{\ast }\left( M\right) $.

\begin{theorem}
\label{ZConditionTheorem}Suppose $Z\circ \sigma _{D_{b}}\left( x,\xi \right)
+\sigma _{D_{b}}\left( x,\xi \right) \mathbf{\circ }Z\ =0$ on $E_{x}$ for
every $x\in M$, and every covector $\xi \in \left( N\overline{\mathcal{F}}%
\right) ^{\ast }$. Then the operator $ZD_{b}+D_{b}Z$ restricts to a bounded
operator on the space of smooth basic sections $\Gamma _{b}\left( M,E,%
\mathcal{F}\right) $.
\end{theorem}

\begin{proof}
Suppose that the hypothesis holds. Let $x\in M$, and using Lemma \ref%
{localAdaptedFrameLemma} choose a local orthonormal frame $e_{1},...,e_{%
\overline{q}},e_{\overline{q}+1},...,e_{q}$ of $N\mathcal{F}$ near $x$ be
such that $e_{1},...,e_{\overline{q}}$ is a local orthonormal frame for the
leaf closure containing $x$, and the fields $e_{1},...,e_{\overline{q}}$
remain tangent to the leaf closures in a neighborhood of $x$. Then 
\begin{eqnarray*}
ZD_{b}+D_{b}Z &=&Z\left( \sum_{i=1}^{q}c\left( e_{i}\right) \nabla
_{e_{i}}^{E}-\frac{1}{2}c\left( \kappa _{b}^{\sharp }\right) \right) +\left(
\sum_{i=1}^{q}c\left( e_{i}\right) \nabla _{e_{i}}^{E}-\frac{1}{2}c\left(
\kappa _{b}^{\sharp }\right) \right) Z \\
&=&\sum_{i=1}^{q}\left( Zc\left( e_{i}\right) +c\left( e_{i}\right) Z\right)
\nabla _{e_{i}}^{E}+\text{bounded} \\
&=&\sum_{i=1}^{\overline{q}}\left( Zc\left( e_{i}\right) +c\left(
e_{i}\right) Z\right) \nabla _{e_{i}}^{E}+\text{bounded}
\end{eqnarray*}%
By Lemma \ref{BoundedOpBasicVFieldLemma}, $\nabla _{e_{i}}^{E}$ is bounded
on the space of basic sections for $i\leq \overline{q}$, and thus $%
ZD_{b}+D_{b}Z$ is a bounded operator on the space of basic sections.
\end{proof}

\begin{remark}
The reader may verify that the converse of the theorem above is also true.
\end{remark}

\begin{remark}
The condition on $Z$ is weaker than the standard nonfoliated case (\cite[%
Section 2]{PrRiPerturb}); $ZD_{b}+D_{b}Z$ can be a first order operator and
still be bounded on $\Gamma _{b}\left( M,E,\mathcal{F}\right) $.
\end{remark}

\begin{remark}
Such a bundle endomorphism $Z$ satisfying this condition does not always
exist. For example, if $M$ is an even dimensional spin manifold, $D$ is a
spin$^{c}$ Dirac operator, and $M$ is foliated by points, then no such $Z$
exists. See \cite[Proposition 2.4]{PrRiPerturb}.
\end{remark}

\subsection{ Proper perturbations of basic Dirac Operators\label%
{properPerturbSection}}

In this section, we state the nondegeneracy conditions on the perturbation.
Note that the perturbations of the basic de Rham operators in \cite{ALMorse}
and \cite{B-P-R} are special cases.

\begin{definition}
Let $Z:E\rightarrow E$ be a smooth basic bundle map on $E\rightarrow \left(
M,\mathcal{F}\right) $.\label{nondegenerate definition}

\begin{enumerate}
\item We say that a leaf closure $\ell $ is \textbf{critical} for $Z$ if $%
Z\left( x\right) =0$ for all $x\in \ell $.

\item We say that a critical leaf closure $\ell $ for $Z$ in $M$ is \textbf{%
nondegenerate} if on any sufficiently small tubular neighborhood $U$ of a
critical leaf closure $\ell $, there exists a constant $c>0$ such that for
all $\alpha \in \Gamma \left( U,E\right) $ and all $y\in U$, 
\begin{equation*}
\left\Vert Z\alpha \right\Vert _{y}\geq c~d\left( y,\ell \right) \left\Vert
\alpha \right\Vert _{y},
\end{equation*}%
where $\left\Vert \cdot \right\Vert _{y}$ is the pointwise norm on $E_{y}$
and where $d\left( y,\ell \right) $ is the distance from $y$ to the leaf
closure $\ell $.
\end{enumerate}
\end{definition}

\begin{definition}
Let $\ell $ be a leaf closure of $\left( M,\mathcal{F}\right) $. Suppose
that a neighborhood $U$ of a point $p\in \ell $ is diffeomorphic to an open
set of the form $U_{1}\times U_{2}$, where $U_{1}=\ell \cap U$ and $U_{2}$
is a open ball in $\mathbb{R}^{\overline{q}}$ centered at the origin. The
corresponding coordinates $\left( x,y\right) $ are called the \textbf{%
adapted coordinates} if each coordinate $y_{j}$ is a locally defined basic
function for $\mathcal{F}$, and $\left\{ \left( x,y\right) :y=0\right\}
=\ell \cap U$. In this case, we say that $U$ is \textbf{adapted to }$\ell $.
\end{definition}

Since $\ell $ is an embedded submanifold of $M$, there exists an adapted
neighborhood of every point in $\ell $.

\begin{lemma}
A critical leaf closure $\ell $ of $\left( M,\mathcal{F}\right) $ is
nondegenerate for a basic bundle map $Z$ if and only if every point of $\ell 
$ has an adapted neighborhood $U$ with adapted coordinates $\left(
x,y\right) $ such that there exist bundle maps $Z_{j}$ for $1\leq j\leq 
\overline{q}$ over $U$ such that $Z=\sum_{j}y_{j}Z_{j}$ on $U$, and $Z$ is
invertible over $U\setminus \ell $.
\end{lemma}

\begin{proof}
Suppose $U$ is an adapted neighborhood of a point of $\ell $, a critical
leaf closure for $Z$. Since $Z$ is smooth and vanishes at each $\left(
x,0\right) $, $Z=\sum_{j}y_{j}Z_{j}$ for some bundle maps $Z_{j}$. Let $%
\overline{U}$ be the closure of $U$; we assume that we have chosen $U$ to be
small enough so that $\overline{U}$ is diffeomorphic to a product of compact
sets in $\ell $ and $\mathbb{R}^{\overline{q}}$. The inequality in the
definition of nondegenerate above is equivalent to 
\begin{equation*}
\left\Vert \sum_{j}\sigma _{j}Z_{j}\alpha \right\Vert _{\left( x,y\right)
}^{2}\geq c^{2}
\end{equation*}%
for every $\sigma \in S^{\overline{q}-1},$ $\alpha \in \Gamma \left(
U,E\right) ,$ $\left( x,y\right) \in \overline{U}$ such that $\left\Vert
\alpha \right\Vert _{\left( x,y\right) }=1$. Since the left hand side of the
inequality is a continuous function of $\sigma $ and $\alpha $ over the
compact set $S^{\overline{q}-1}\times \left\{ \alpha \in \Gamma \left( 
\overline{U},E\right) \,:\left\Vert \alpha \right\Vert _{\left( x,y\right)
}=1\text{ for all }\left( x,y\right) \in U\right\} $, its infimum is
attained. It follows that on each such neighborhood, $Z$ is invertible away
from $\ell $ if and only if the inequality holds with $c>0$. Compactness of $%
\ell $ implies the inequality holds in tubular neighborhoods of $\ell $.
\end{proof}

\begin{definition}
\label{properPerturbationDef}Let $D_{b}^{\pm }:\Gamma \left( M,E^{\pm
}\right) \rightarrow \Gamma \left( M,E^{\mp }\right) $ be the basic Dirac
operator associated to a bundle of graded Clifford modules. Let $%
D_{s}=D_{b}+sZ$ for $s\in \mathbb{R}$, where $Z=\left( Z^{+},\left(
Z^{+}\right) ^{\ast }\right) \in \Gamma _{b}\left( M,\mathrm{End}\left(
E^{+},E^{-}\right) \right) $. We say that $Z$ is a \textbf{proper
perturbation of }$D_{b}$ if

\begin{enumerate}
\item $\left( D_{s}\right) ^{2}-D_{b}^{2}$ \ is a operator that is bounded
on $\Gamma _{b}\left( M,E,\mathcal{F}\right) $.

\item All critical leaf closures for $Z$ are nondegenerate.
\end{enumerate}
\end{definition}

\begin{remark}
\vspace{0in}The first condition is satisfied if $Z\circ \sigma
_{D_{b}}\left( x,\xi \right) +\sigma _{D_{b}}\left( x,\xi \right) \mathbf{%
\circ }Z\ =0$ on $E_{x}$ for every $x\in M$, and every covector $\xi \in
\left( N\overline{\mathcal{F}}\right) ^{\ast }$; see Theorem \ref%
{ZConditionTheorem}.
\end{remark}

The following lemma shows that only certain ranks of vector bundles may
admit proper perturbations.

\begin{lemma}
\label{dimensionLemma}(in \cite{PrRiPerturb})For any positive integer $k$,
there exists a linear map $L:\mathbb{R}^{k}\rightarrow M_{r}\left( \mathbb{C}%
\right) $ that satisfies $L\left( x\right) \in \mathrm{Gl}\left( r,\mathbb{C}%
\right) $ for $x\neq 0$ if and only if $r=m2^{\left\lfloor \frac{k-1}{2}%
\right\rfloor }$ for some positive integer $m$.
\end{lemma}

The following result shows that nonsingular proper perturbations always
exist on Clifford modules over a Riemannian foliation of odd codimension.

\begin{proposition}
\label{OddAlwaysExists}Suppose that the codimension $q$ of $\left( M,%
\mathcal{F}\right) $ is odd. Let $E$ be a bundle of basic graded Clifford
modules over $M$, and let $D_{b}$ be the corresponding basic Dirac operator.
Then there always exists a proper basic perturbation $Z$ of $D_{b}$; in
particular the perturbation may be chosen to be invertible.
\end{proposition}

\begin{proof}
Near a point of $M$, let $\left( e_{1},...,e_{q}\right) $ be a local basic
orthonormal frame of the normal bundle. Let $Z=i^{q\left( q+1\right)
/2}c\left( e_{1}\right) c\left( e_{2}\right) ...c\left( e_{q}\right) $,
which is well-defined independent of the choice of orthonormal frame, and it
is a basic bundle map from $E^{\pm }$ to $E^{\mp }$. It satisfies the
conditions of Theorem \ref{ZConditionTheorem}. Since the adjoint of $Z$ is 
\begin{equation*}
Z^{\ast }=i^{q\left( q+1\right) /2}\left( -1\right) ^{q\left( q+1\right)
/2}\left( -1\right) ^{q}c\left( e_{q}\right) c\left( e_{q-1}\right)
...c\left( e_{1}\right) =i^{q\left( q+1\right) /2}c\left( e_{1}\right)
c\left( e_{2}\right) ...c\left( e_{q}\right) =Z,
\end{equation*}
we have a proper basic perturbation of $D_{b}$ that is globally invertible.
\end{proof}

\section{Vector bundles over foliations and compact group actions\label%
{MolinoSection}}

In this section, we restrict to a subset of foliated vector bundles over $%
\left( M,\mathcal{F}\right) $. We will discuss $\mathcal{F}$\textbf{%
-equivariant vector bundles}; most of the interesting cases of geometrically
constructed foliated bundles are also $\mathcal{F}$-equivariant.

\vspace{1pt}Let $\left( M,\mathcal{F},g_{Q}\right) $ be a Riemannian
foliation with bundle-like metric $g$, let $N\mathcal{F}=\left( T\mathcal{F}%
\right) ^{\bot }\cong Q$. Let $M\diagup \overline{\mathcal{F}}$ denote the
space of leaf closures.

\begin{definition}
The bundle $E\rightarrow M$ is called an $\mathcal{F}$-vector bundle if the
following holonomy lifting property is satisfied. For any holonomy
diffeomorphism $h:U\rightarrow V$ associated to a path $\gamma $ in a leaf
from $x\in D$ to $y\in D$, where $D$ is a submanifold at $x$ transverse to $%
\mathcal{F}$ and $U$ and $V$ are open subsets of $D$ such that $x,y\in U\cap
V$, $h$ lifts to a bundle isomorphism $\widetilde{h}:\left. E\right\vert
_{U}\rightarrow \left. E\right\vert _{V}$, and all such maps satisfy

\begin{enumerate}
\item $\widetilde{h_{1}h_{2}}=\widetilde{h_{1}}\widetilde{h_{2}}$ (when both
sides are defined)

\item $\widetilde{\mathrm{id}_{D}}=\mathrm{id}_{E}$

\item $\widetilde{\left. h\right\vert _{W}}=\left. \widetilde{h}\right\vert
_{W}$ for any proper open subset $W$ of the domain of $h$.
\end{enumerate}
\end{definition}

Such a vector bundle is often called \textbf{equivariant with respect to the
foliation groupoid}, and these vector bundles are examples of foliated
vector bundles (see \cite{KT2}), as described in Section \ref%
{preliminaryperturb}. They come equipped with a partial connection $\nabla
^{E}$, which may be completed to a basic connection (with the same name).
When $E$ is an $\mathcal{F}$-vector bundle, the basic sections are those
sections that satisfy $\widetilde{h}\left( s_{x}\right) =s_{y}$ for every
holonomy diffeomorphism $h$ as above with $h\left( x\right) =y$, i.e. the
holonomy-invariant sections. Let $\mathrm{Diff}_{b}\left( M,E,\mathcal{F}%
\right) $ denote the space of differential operators that preserve the basic
sections, and let $\mathrm{Diff}_{b}^{\ast }\left( M,E,\mathcal{F}\right) $
denote the space of their restrictions to $\Gamma _{b}\left( M,E,\mathcal{F}%
\right) $. In the above, we call an $\mathcal{F}$-vector bundle an Hermitian 
$\mathcal{F}$-vector bundle if it is equipped with a holonomy-invariant
Hermitian inner product.

Let $\pi :\widehat{M}\rightarrow M$ denote the bundle of orthonormal
transverse frames over $\left( M,\mathcal{F},g_{Q}\right) $. Let $\pi ^{\ast
}\mathcal{F}$ be the pullback foliation, which has dimension $\mathrm{\dim }%
\left( \mathcal{F}\right) +\mathrm{\dim }\left( O\left( q\right) \right) $.
The connection on $\widehat{M}$ is flat along the leaves, since $\left( M,%
\mathcal{F},g_{Q}\right) $ is Riemannian. The lifted foliation $\widehat{%
\mathcal{F}}$ is defined on $\widehat{M}$ by horizontally lifting the leaves
of $\mathcal{F}$ locally. The tangent space of $\widehat{\mathcal{F}}$ is
the intersection of the tangent space to the pullback foliation and the
horizontal space coming from the adapted connection. The lifted foliation $%
\widehat{\mathcal{F}}$ has the same dimension as $\mathcal{F}$ and is
transversally parallelizable, and this implies that the leaf closure space $%
W=\widehat{M}\diagup \overline{\widehat{\mathcal{F}}}$ is a manifold, called
the \textbf{basic manifold}\emph{\ }(see \cite{Mo}). The quotient map $p:%
\widehat{M}\rightarrow W$ is a Riemannian submersion for a natural choice of
metric on $\widehat{M}${\LARGE . }For any $O\left( q\right) $-equivariant
vector bundle $F\rightarrow W$, let $\mathrm{Diff}_{O\left( q\right) }\left(
W,F\right) $ denote the space of $O\left( q\right) $-equivariant
differential operators, and let $\mathrm{Diff}_{O\left( q\right) }^{\ast
}\left( W,F\right) $ denote the space of their restrictions to $O\left(
q\right) $-invariant sections of $F$.

For each Hermitian $\mathcal{F}$-vector bundle $\pi :E\rightarrow M$, we
construct a canonically associated Hermitian $O\left( q\right) $-vector
bundle $\pi ^{\prime }:E^{\prime }\rightarrow W$ that has similar geometric
and analytic properties, as follows. Specifically, the bundle $E^{\prime }$
over $W$ is defined by 
\begin{equation*}
E_{w}^{\prime }=\Gamma _{b}\left( p^{-1}\left( w\right) ,\left. \pi ^{\ast
}E\right\vert _{p^{-1}\left( w\right) },\left. \widehat{\mathcal{F}}%
\right\vert _{p^{-1}\left( w\right) }\right) ,
\end{equation*}%
and it is a $O\left( q\right) $-equivariant Hermitian vector bundle of
finite rank (see \cite[section 2.7]{EK}) . It is also true that $p^{\ast
}E^{\prime }=\pi ^{\ast }E$, and we have the algebra isomorphisms 
\begin{eqnarray*}
\Gamma _{b}\left( M,E,\mathcal{F}\right) &\longrightarrow &\Gamma _{b}\left( 
\widehat{M},\pi ^{\ast }E,\pi ^{\ast }\mathcal{F}\right) \\
&\longrightarrow &\Gamma \left( W,E^{\prime }\right) ^{O\left( q\right) }.
\end{eqnarray*}%
The isomorphisms are determined by the correspondences%
\begin{eqnarray*}
s &\in &\Gamma _{b}\left( M,E,\mathcal{F}\right) \longleftrightarrow \text{%
section }\widehat{s}=\pi ^{\ast }s\in \Gamma _{b}\left( \widehat{M},\pi
^{\ast }E,\pi ^{\ast }\mathcal{F}\right) \\
&\longleftrightarrow &\left. \widehat{s}\right\vert _{p^{-1}\left( w\right)
}\in \Gamma _{b}\left( p^{-1}\left( w\right) ,\left. \pi ^{\ast
}E\right\vert _{p^{-1}\left( w\right) },\left. \widehat{\mathcal{F}}%
\right\vert _{p^{-1}\left( w\right) }\right) \text{ for each }w\in W \\
&\longleftrightarrow &\widetilde{s}\in \Gamma \left( W,E^{\prime }\right)
^{O\left( q\right) },\widetilde{s}\left( w\right) :=\left. \widehat{s}%
\right\vert _{p^{-1}\left( w\right) }
\end{eqnarray*}

Given a differential operator $P\in \mathrm{Diff}_{b}^{\ast }\left( M,E,%
\mathcal{F}\right) $, it induces an operator $\widetilde{P}$ on $\Gamma
\left( W,E^{\prime }\right) ^{O\left( q\right) }$ by%
\begin{equation*}
\widetilde{P}\left( \widetilde{s}\right) \left( w\right) =\left. \pi ^{\ast
}\left( Ps\right) \right\vert _{p^{-1}\left( w\right) },
\end{equation*}%
where $s$ is the unique section in $\Gamma _{b}\left( M,E,\mathcal{F}\right) 
$ such that $\widetilde{s}\left( w\right) =\left. \pi ^{\ast }s\right\vert
_{p^{-1}\left( w\right) }$. The following proposition from \cite{ALMorse}
gives the important properties of this construction.

\begin{proposition}
(\cite[variant of Theorem 2.1]{ALMorse}) \label%
{OperatorIsomorphismProposition}For each Hermitian $\mathcal{F}$-vector
bundle $\pi :E\rightarrow M$, there is a canonically associated Hermitian $%
O\left( q\right) $-vector bundle $\pi ^{\prime }:E^{\prime }\rightarrow W$
and a canonical isomorphism of algebras%
\begin{equation*}
\mathrm{Diff}_{b}^{\ast }\left( M,E,\mathcal{F}\right) \cong \mathrm{Diff}%
_{O\left( q\right) }^{\ast }\left( W,E^{\prime }\right) .
\end{equation*}%
The isomorphism preserves transverse ellipticity, and there is a metric on $%
W $ so that the isomorphism preserves formal adjoints with respect to the $%
L^{2}$ inner products.
\end{proposition}

\begin{corollary}
The previous proposition is also valid locally. Namely, if $M$ is replaced
with an $\varepsilon $-tubular neighborhood $T_{\varepsilon }$ of a leaf
closure $\ell \subset M$ and $W$ is replaced by a tubular neighborhood $%
T_{\varepsilon }^{\prime }$ of the $O\left( q\right) $-orbit $p\left( \pi
^{-1}\left( \ell \right) \right) $, then $\mathrm{Diff}_{b}^{\ast }\left(
T_{\varepsilon },\left. E\right\vert _{T_{\varepsilon }},\mathcal{F}\right)
\cong \mathrm{Diff}_{O\left( q\right) }^{\ast }\left( T_{\varepsilon
}^{\prime },\left. E^{\prime }\right\vert _{T_{\varepsilon }^{\prime
}}\right) $.
\end{corollary}

Let $C_{b}^{\infty }\left( M,\mathcal{F}\right) $ denote the space of smooth
basic functions, and let $C^{\infty }\left( W\right) ^{O\left( q\right) }$
denote the space of smooth $O\left( q\right) $-invariant functions. The
following proposition is similar to \cite[Propositions 3.3 and 3.4]{ALMorse}%
, which considers the special case of differential forms. As above, let $%
\mathrm{Diff}_{G}\left( T_{\varepsilon }^{\prime },\left. E^{\prime
}\right\vert _{T_{\varepsilon }^{\prime }}\right) $ be the set of $G$%
-equivariant differential operators on $\Gamma \left( T_{\varepsilon
}^{\prime },\left. E^{\prime }\right\vert _{T_{\varepsilon }^{\prime
}}\right) $. Two such operators are called \textbf{equivalent} if their
restrictions to $G$-invariant sections are identical. Let $\mathrm{Diff}%
_{G}^{\ast }\left( T_{\varepsilon }^{\prime },\left. E^{\prime }\right\vert
_{T_{\varepsilon }^{\prime }}\right) $ be the set of restrictions of
elements of $\mathrm{Diff}_{G}\left( T_{\varepsilon }^{\prime },\left.
E^{\prime }\right\vert _{T_{\varepsilon }^{\prime }}\right) $ to $G$%
-invariant sections; note that equivalent $G$-equivariant operators yield a
single element of $\mathrm{Diff}_{G}^{\ast }\left( T_{\varepsilon }^{\prime
},\left. E^{\prime }\right\vert _{T_{\varepsilon }^{\prime }}\right) $.

\begin{proposition}
\label{tubeProposition}Let $T_{\varepsilon }^{\prime }$ be a tubular
neighborhood of an orbit $\mathcal{O}_{x}$ of a point $x$ in a Riemannian $G$%
-manifold $W$. Let $B_{\varepsilon }$ be the image of $\varepsilon $-ball in
the normal space $\left( T_{x}\mathcal{O}_{x}\right) ^{\bot }$ under the
exponential map at $x$. Then there is a canonical isomorphism of algebras%
\begin{equation*}
\mathrm{Diff}_{G}^{\ast }\left( T_{\varepsilon }^{\prime },\left. E^{\prime
}\right\vert _{T_{\varepsilon }^{\prime }}\right) \cong \mathrm{Diff}%
_{H}^{\ast }\left( B_{\varepsilon },\left. E^{\prime }\right\vert
_{B_{\varepsilon }}\right) ,
\end{equation*}%
where $H$ is the isotropy subgroup of $G$ at $x$. Further, the isomorphism
puts transversally elliptic operators in $\mathrm{Diff}_{G}^{\ast }\left(
T_{\varepsilon }^{\prime },\left. E^{\prime }\right\vert _{T_{\varepsilon
}^{\prime }}\right) $ in one-to-one correspondence with elliptic operators
in $\mathrm{Diff}_{H}^{\ast }\left( B_{\varepsilon },\left. E^{\prime
}\right\vert _{B_{\varepsilon }}\right) $. In addition, there is a
Riemannian metric on $B_{\varepsilon }$ such that isomorphism preserves
formal adjoints with respect to $L^{2}$ inner products. Finally, if $P^{\ast
}\in \mathrm{Diff}_{G}^{\ast }\left( T_{\varepsilon }^{\prime },\left.
E^{\prime }\right\vert _{T_{\varepsilon }^{\prime }}\right) $ and $P^{\prime
}$ is the corresponding element of $\mathrm{Diff}_{H}^{\ast }\left(
B_{\varepsilon },\left. E^{\prime }\right\vert _{B_{\varepsilon }}\right) $,
and if $s\in \Gamma \left( T_{\varepsilon }^{\prime },\left. E^{\prime
}\right\vert _{T_{\varepsilon }^{\prime }}\right) ^{G}$, then%
\begin{equation*}
\left. P\left( s\right) \right\vert _{B_{\varepsilon }}=P^{\prime }\left(
\left. s\right\vert _{B_{\varepsilon }}\right) .
\end{equation*}
\end{proposition}

\begin{proof}
First, note that sections in $\Gamma \left( T_{\varepsilon }^{\prime
},\left. E^{\prime }\right\vert _{T_{\varepsilon }^{\prime }}\right) ^{G}$
are in one-to-one correspondence with sections of $\Gamma \left(
B_{\varepsilon },\left. E^{\prime }\right\vert _{B_{\varepsilon }}\right)
^{H}$. To see this, given a section $s\in \Gamma \left( T_{\varepsilon
}^{\prime },\left. E^{\prime }\right\vert _{T_{\varepsilon }^{\prime
}}\right) ^{G}$, the restriction $\Phi \left( s\right) =\left. s\right\vert
_{B_{\varepsilon }}$ is an element of $\Gamma \left( B_{\varepsilon },\left.
E^{\prime }\right\vert _{B_{\varepsilon }}\right) ^{H}$. Note that $\Phi $
is a linear map that is one-to-one, because each section of $\Gamma \left(
T_{\varepsilon }^{\prime },\left. E^{\prime }\right\vert _{T_{\varepsilon
}^{\prime }}\right) ^{G}$ is determined by its restriction to $%
B_{\varepsilon }$, because $B_{\varepsilon }$ intersects every orbit in $%
T_{\varepsilon }^{\prime }$. Next, if $u\in \Gamma \left( B_{\varepsilon
},\left. E^{\prime }\right\vert _{B_{\varepsilon }}\right) ^{H}$, $u$
extends to $T_{\varepsilon }^{\prime }$ by $u^{\prime }\left( gx\right)
:=g~u\left( x\right) $ for $g\in G$ and $x\in B_{\varepsilon }$. Note that
if $g_{1}x=g_{2}x$, then $g_{2}^{-1}g_{1}x=x$, so that $g_{2}^{-1}g_{1}$ is
an element of the isotropy subgroup at $x$. Thus $u^{\prime }\left(
g_{2}x\right) =g_{2}u\left( x\right) =g_{2}u\left( g_{2}^{-1}g_{1}x\right)
=g_{1}u\left( x\right) =u^{\prime }\left( g_{1}x\right) $, so that the
extension is well-defined. Thus, $\Phi :\Gamma \left( T_{\varepsilon
}^{\prime },\left. E^{\prime }\right\vert _{T_{\varepsilon }^{\prime
}}\right) ^{G}\rightarrow \Gamma \left( B_{\varepsilon },\left. E^{\prime
}\right\vert _{B_{\varepsilon }}\right) ^{H}$ is an isomorphism.

Let $S$ be a small open ball complementary and transverse to the $H$-orbits
through the identity in $G$. Let $\ell _{x}=\left\{ sx:s\in S\right\} $ for $%
x\in B_{\varepsilon }$. Note that $\ell _{x}\cap \ell _{y}=\emptyset $ if $%
x\neq y$ and $x,y\in B_{\varepsilon }$; if $s_{1}x=$ $s_{2}y$ with $%
s_{1},s_{2}\in S$ implies $s_{2}^{-1}s_{1}x=y$, so that $s_{2}^{-1}s_{1}\in H
$, $s_{1}\in s_{2}H$, so that $s_{1}=s_{2}$ since $S$ is transverse to the $H
$-orbits in $G$. The leaves $\left\{ \ell _{x}:x\in B_{\varepsilon }\right\} 
$ form a trivial Riemannian foliation $\mathcal{L}$ of a neighborhood $N$ of 
$B_{\varepsilon }$ in $T_{\varepsilon }^{\prime }$. Note that the
restriction of $\Gamma \left( T_{\varepsilon }^{\prime },\left. E^{\prime
}\right\vert _{T_{\varepsilon }^{\prime }}\right) ^{G}$ to $N$ is a subspace
of the space of basic sections of $\left. E^{\prime }\right\vert _{N}$,
which is a foliated vector bundle over $\left( N,\mathcal{L}\right) $. Note
that $u$ is a basic section if and only if $\nabla _{X}^{E^{\prime }}u=0$
for all leafwise vector fields $X$.

Given an operator $P^{\ast }\in \mathrm{Diff}_{G}^{\ast }\left(
T_{\varepsilon }^{\prime },\left. E^{\prime }\right\vert _{T_{\varepsilon
}^{\prime }}\right) $, it is represented by a $G$-equivariant differential
operator $P$ in $\mathrm{Diff}_{G}\left( T_{\varepsilon }^{\prime },\left.
E^{\prime }\right\vert _{T_{\varepsilon }^{\prime }}\right) $. Then $\left.
P\right\vert _{\Gamma \left( N,E^{\prime }\right) }$ is a differential
operator that maps $\Gamma _{b}\left( N,E^{\prime }\right) $, the space of
basic sections $E^{\prime }$ over $N$, to itself, because for $g\in S$, and $%
u\in \Gamma \left( T_{\varepsilon }^{\prime },\left. E^{\prime }\right\vert
_{T_{\varepsilon }^{\prime }}\right) $, $P\left( g\cdot u\right) =g\cdot
P\left( u\right) $. We choose a framing of $TN$ adapted to $\mathcal{L}$: $%
\left( e_{1},...,e_{p},e_{p+1},...,e_{n}\right) $, where $\left(
e_{1},...,e_{p}\right) $ is tangent to $\mathcal{L}$ and $\left(
e_{p+1},...,e_{n}\right) $ is perpendicular to $\mathcal{L}$. Then we may
write $\left. P\right\vert _{\Gamma _{b}\left( N,E^{\prime }\right) }$ as a
polynomial over $C^{\infty }\left( N\right) $ in the covariant derivatives $%
\nabla _{e_{j}}^{E^{\prime }}$, such that 
\begin{equation*}
P=P_{1}+\sum_{j=1}^{p}P_{2,j}\nabla _{e_{j}}^{E^{\prime }}~,
\end{equation*}%
where $P_{1}$ is a differential operator whose only derivatives are of the
form $\nabla _{e_{j}}^{E^{\prime }}$ with $j>p$. Here we have used the fact
that operators of the form $\nabla _{a}^{E^{\prime }}\nabla _{b}^{E^{\prime
}}-\nabla _{b}^{E^{\prime }}\nabla _{a}^{E^{\prime }}$ are first-order
differential operators for $a$ and $b$ coordinate vector fields. Thus the
restriction satisfies%
\begin{equation*}
\left. P\right\vert _{\Gamma _{b}\left( N,E^{\prime }\right) }=\left.
P_{1}\right\vert _{\Gamma _{b}\left( N,E^{\prime }\right) }.
\end{equation*}%
Thus, $P$ restricted to basic sections may be expressed as a differential
operator on the local quotient $N\diagup \mathcal{L}$, which is
diffeomorphic to $B_{\varepsilon }$. Since $P$ is $H$-equivariant, the
corresponding operator on $B_{\varepsilon }$ must be $H$-equivariant. If we
let $P^{\prime }$ be that operator, the proof is complete.
\end{proof}

\begin{remark}
If $E$ comes equipped with a grading, a similar grading is induced on $%
E^{\prime }$, and the constructions and isomorphisms above preserve those
gradings.
\end{remark}

\section{Localization and consequences}

\subsection{Localization to the model operator\label{ShubinLocalization}}

In this section $G=O\left( q\right) $. We start with an operator of the form 
$D_{b}+sZ\in \mathrm{Diff}_{b}^{\ast }\left( M,E,\mathcal{F}\right) $, where 
$D_{b}$ is a basic Dirac operator and $Z$ is an $\mathcal{F}$-basic bundle
map. Let $D^{\prime }+sZ^{\prime }\in \mathrm{Diff}_{O\left( q\right)
}^{\ast }\left( W,E^{\prime }\right) $ be the $O\left( q\right) $%
-equivariant operator on $W$ corresponding to $D_{b}+sZ$ via the isomorphism
in Proposition \ref{OperatorIsomorphismProposition}. The critical leaf
closures of $Z$ on $M$ correspond exactly to critical orbits of $Z^{\prime }$
on $W$.

Let 
\begin{equation*}
H_{s}=\frac{1}{s}\left( D^{\prime }+sZ^{\prime }\right) ^{2}.
\end{equation*}%
We wish to study the asymptotics of the $O\left( q\right) $-invariant part
of the spectrum of $H_{s}$ in the semiclassical limit as $s\rightarrow
\infty $. The goal of this section is to construct the model operator, whose
spectrum approximates the spectrum of $H_{s}$, under the following
conditions:

\begin{enumerate}
\item $\left( D^{\prime }\right) ^{2}$ is a second order, transversally
elliptic differential operator with positive definite principal transverse
symbol.

\item $\left. \left( D^{\prime }Z^{\prime }+Z^{\prime }D^{\prime }\right)
\right\vert _{\Gamma \left( W,E^{\prime }\right) ^{G}}$ is a bounded
operator.

\item The bundle map $\left( Z^{\prime }\right) ^{2}$ satisfies $\left(
Z^{\prime }\left( x\right) \right) ^{2}\geq 0$ for all $x\in M$, and at each
point $\overline{x}$ where $Z^{\prime }\left( \overline{x}\right) $ is
singular, we have $Z^{\prime }\left( \overline{x}\right) =0$, $\,$and there
exists $c>0$ such that 
\begin{equation*}
\left( Z^{\prime }\left( x\right) \right) ^{2}\geq c\cdot d\left( \mathcal{O}%
_{x},\overline{x}\right) ^{2}\mathbf{1}
\end{equation*}%
in a neighborhood of $\overline{x}$, where $d\left( \mathcal{O}_{x},%
\overline{x}\right) $ is the distance between the orbit $\mathcal{O}_{x}$ of 
$x$ to $\overline{x}$.
\end{enumerate}

\begin{remark}
These conditions are satisfied for $D^{\prime }$ and $Z^{\prime }$ if an
only if $\ Z$ is a proper perturbation of $D_{b}$, where in particular $%
Z\circ \sigma _{D_{b}}\left( x,\xi \right) +\sigma _{D_{b}}\left( x,\xi
\right) \mathbf{\circ }Z\ =0$ on $E_{x}$ for every $x\in M$, and every
covector $\xi \in \left( N\overline{\mathcal{F}}\right) ^{\ast }$. See
Theorem \ref{ZConditionTheorem}.
\end{remark}

The singular set of $Z^{\prime }$ is a union of isolated orbits; each such
critical orbit corresponds to a critical leaf closure of $Z$. We will choose
an arbitrary point $\overline{x}$ on each critical orbit, labeled $\mathcal{O%
}_{\overline{x}}$. Then the model operator is a direct sum $K^{G}=\bigoplus_{%
\text{critical }\mathcal{O}_{\overline{x}}}K_{\overline{x}}^{H_{\overline{x}%
}}$, where the operators $K_{\overline{x}}^{H_{\overline{x}}}$ are
constructed below. Because all operators are $O\left( q\right) $%
-equivariant, the spectrum of $K_{\overline{x}}^{H_{\overline{x}}}$ is
independent of the choice of $\overline{x}$ on a fixed orbit.

At each $\overline{x}$ where $Z^{\prime }\left( \overline{x}\right) $ is
singular, the corresponding operators $\left( H_{s}\right) ^{\prime }$ from
Proposition \ref{tubeProposition} are $G_{\overline{x}}$ equivariant on the
transverse balls $B_{\varepsilon ,\overline{x}}$, where $G_{\overline{x}}$
is the isotropy subgroup at $\overline{x}$. Note that the quadratic
condition (3) on $Z^{\prime }$ is equivalent to the condition of
nondegeneracy for a leaf closure $\ell $ that is critical for $Z$ in
Definition \ref{nondegenerate definition}.

Let $H_{s}^{G}$ denote the restriction of $H_{s}$ to $\Gamma \left(
W,E^{\prime }\right) ^{G}$. It is known (see \cite[p. 12-13]{A}) that the
operator $H_{s}^{G}$ has discrete spectrum.

Near each critical orbit $\mathcal{O}_{\overline{x}}$ of $\left( Z^{\prime
}\right) ^{2}$, we choose coordinates $x=$ $\left( x_{1},...,x_{m}\right) $
for the transverse ball $B_{\varepsilon ,\overline{x}}$ such that $\overline{%
x}$ corresponds to the origin, $N_{\overline{x}}\mathcal{O}_{\overline{x}}=%
\mathbb{R}^{m}$, and the volume form at the origin is $dx_{1}\wedge
...\wedge dx_{m}$. Let $H_{\overline{x}}$ be the isotropy subgroup at $%
\overline{x}$. We choose a trivialization of $E^{\prime }$ near $\overline{x}
$. Then $A=\left( D^{\prime }\right) ^{2}$, $B=D^{\prime }Z^{\prime
}+Z^{\prime }D^{\prime }$, and $C=\left( Z^{\prime }\right) ^{2}$ become
differential operators with matrix coefficients. We define the model
operator $K_{\overline{x}}^{H_{\overline{x}}}:\Gamma \left( \mathbb{R}%
^{m},E_{\overline{x}}^{\prime }\right) ^{H_{\overline{x}}}\rightarrow \Gamma
\left( \mathbb{R}^{m},E_{\overline{x}}^{\prime }\right) ^{H_{\overline{x}}}$
by 
\begin{eqnarray*}
K_{\overline{x}}^{H_{\overline{x}}} &=&\widetilde{A}+\widetilde{B}+%
\widetilde{C},\text{ where} \\
\widetilde{A} &=&\text{the principal part of }A\text{ at }\overline{x} \\
\widetilde{B} &=&\left. B\right\vert _{\Gamma \left( W,E^{\prime }\right)
^{G}}\left( \overline{x}\right) \\
\widetilde{C} &=&\sum x_{i}x_{j}\left( \nabla _{i}\nabla _{j}C\right) _{%
\overline{x}}=\text{the quadratic part of }C\text{ at }\overline{x},
\end{eqnarray*}%
where $\nabla $ is the induced connection on $E^{\prime }\otimes E^{\prime
\ast }$. It is easy to check that $\widetilde{C}$ is independent of the
coordinates and connection chosen. Let $dg$ denote the differential of the
action of $g\in H_{\overline{x}}$ at $\overline{x}$, so we write $dg:\mathbb{%
R}^{m}\rightarrow \mathbb{R}^{m}$. Let the action of $g$ on $\mathbb{R}%
^{m}\times E_{\overline{x}}^{\prime }$ be defined as 
\begin{equation*}
\left( x,v_{\overline{x}}\right) g=\left( dg\left( x\right) ,g\cdot v_{%
\overline{x}}\right) .
\end{equation*}

\begin{lemma}
The operator $K_{\overline{x}}^{H_{\overline{x}}}$ is equivariant with
respect to this $H_{\overline{x}}$-action.
\end{lemma}

\begin{proof}
Since $H_{s}$ is equivariant with respect to $G$ for each $s>0$, it is easy
to show that each of the operators $A$, $B$, and $C$ is $G$-equivariant.
Then the principal symbol of $A$ is $G$-equivariant, and in particular the
principal symbol of $A$ at $\overline{x}$ is $H_{\overline{x}}$-invariant.
Thus, $\widetilde{A}$ is $H_{\overline{x}}$-invariant. Next, since $\left.
B\right\vert _{\Gamma \left( W,E^{\prime }\right) ^{G}}$ is equivariant, its
restriction $\widetilde{B}$ to $\overline{x}$ is also. Finally, since $C$ is 
$G$-equivariant and the connection is $G$-equivariant, it follows that $%
\widetilde{C}$ is $H_{\overline{x}}$-equivariant.
\end{proof}

\begin{lemma}
\label{discreteSpectrumLemma}The operator $K_{\overline{x}}^{H_{\overline{x}%
}}:\Gamma \left( \mathbb{R}^{m},E_{\overline{x}}^{\prime }\right) ^{H_{%
\overline{x}}}\rightarrow \Gamma \left( \mathbb{R}^{m},E_{\overline{x}%
}^{\prime }\right) ^{H_{\overline{x}}}$ has discrete spectrum.
\end{lemma}

\begin{proof}
Consider the extended operator $K_{\overline{x}}^{H_{\overline{x}}}:\Gamma
\left( \mathbb{R}^{m},E_{\overline{x}}^{\prime }\right) \rightarrow \Gamma
\left( \mathbb{R}^{m},E_{\overline{x}}^{\prime }\right) $ . This operator is
elliptic and essentially self-adjoint, and the operator is bounded below by $%
\left( C_{1}+C_{2}\cdot \left\vert x\right\vert ^{2}\right) \mathbf{1}$,
where $C_{1}\in \mathbb{R}$ and $C_{2}>0$. Since this bound goes to infinity
as $x\rightarrow \pm \infty $, the operator $K_{\overline{x}}^{H_{\overline{x%
}}}-\left( C_{1}-1\right) \mathbf{1}$ has a compact resolvent. Thus, the
restriction of $K_{\overline{x}}^{H_{\overline{x}}}$ to $\Gamma \left( 
\mathbb{R}^{m},E_{\overline{x}}^{\prime }\right) ^{H_{\overline{x}}}$ also
has a compact resolvent.
\end{proof}

\begin{remark}
\vspace{0in}The dimension $m$ above may depend on $\overline{x}$, even
though this is not obvious from the notation.
\end{remark}

We define the model operator $K^{G}$ by 
\begin{equation*}
K^{G}=\bigoplus_{\text{critical }\mathcal{O}_{\overline{x}}}K_{\overline{x}%
}^{H_{\overline{x}}}.
\end{equation*}%
Clearly, this operator has discrete spectrum.

\vspace{1pt}Let $\lambda _{1}^{G}\left( s\right) \leq \lambda _{2}^{G}\left(
s\right) \leq ...$ be the eigenvalues of $H_{s}^{G}$, repeated according to
multiplicity, correspond to the orthonormal basis of eigensections $\omega
_{1}^{G}\left( s\right) ,\omega _{2}^{G}\left( s\right) ,...$ . Let $\mu
_{1}^{G}\leq \mu _{2}^{G}\leq ...$ be the eigenvalues of the model operator $%
K^{G}$, repeated according to multiplicity, correspond to the $L^{2}$%
-orthonormal basis of eigensections $\phi _{1,s}^{G},\phi _{1,s}^{G}$, ....
Then we have the following result.

\begin{theorem}
(Equivariant Localization Theorem)\label{EquLocalizationThm copy(1)}Assume
that the singular set of $Z$ is not empty. Then, for each fixed $N>0$, there
exists $C>0$ and $s_{0}>0$ such that for any $s>s_{0}$ and any $j\leq N$, $%
\left\vert \lambda _{j}^{G}\left( s\right) -\mu _{j}^{G}\right\vert \leq
Cs^{-1/5}$. If the singular set of $Z$ is empty, then there is a $c>0$ such
that for $s$ sufficiently large, $\lambda _{1}^{G}\left( s\right) \geq cs$.
\end{theorem}

\begin{proof}
This proof is a generalization of Theorem 1.1 in \cite{Shu1} to the
equivariant setting. We identify the parameter $s$ in our theorem with $%
\frac{1}{h}$ in \cite{Shu1}.

To obtain an upper bound for the eigenvalues of $H_{s}^{G}$ (or a lower
bound on the spectral counting function of $H_{s}^{G}$), we use
eigensections of the model operator $K^{G}$ to produce test sections for $%
H_{s}^{G}$ in the Rayleigh quotient. Suppose that $\psi $ is an eigensection
of $K_{\overline{x}}^{H_{\overline{x}}}:\Gamma \left( \mathbb{R}^{m},E_{%
\overline{x}}^{\prime }\right) ^{H_{\overline{x}}}\rightarrow \Gamma \left( 
\mathbb{R}^{m},E_{\overline{x}}^{\prime }\right) ^{H_{\overline{x}}}$
corresponding to the eigenvalue $\lambda $. Let $J\in C_{0}^{\infty }\left( 
\mathbb{R}^{m}\right) $ be a radial function defined such that $0\leq J\leq
1 $ , $J\left( x\right) =1$ if $\left\vert x\right\vert \leq 1$, $J\left(
x\right) =0$ if $\left\vert x\right\vert \geq 2$. For any $s>0$, let $%
J^{\left( s\right) }\left( x\right) =J\left( s^{2/5}x\right) $. Then the
section 
\begin{equation*}
\phi \left( x\right) =J^{\left( s\right) }\left( x\right) s^{n/2}\psi \left(
s^{1/2}x\right)
\end{equation*}%
is in $\Gamma \left( \mathbb{R}^{m},E_{\overline{x}}^{\prime }\right) ^{H_{%
\overline{x}}}$ as well, because $J^{\left( s\right) }$ is $G$-invariant. We
produce a corresponding element $\widetilde{\phi }\in \Gamma \left(
B_{\varepsilon ,\overline{x}},E^{\prime }\right) ^{H_{\overline{x}}}$ that
has support in a small neighborhood $B_{\varepsilon ,\overline{x}}$ of $%
\overline{x}$, as follows. Let $\gamma $ be the unit speed geodesic from $%
\overline{x}$ to $p\in B_{\varepsilon ,\overline{x}}$, let $x_{p}$ be the
geodesic normal coordinates of $p$, and let $P_{\gamma }:E_{\overline{x}%
}^{\prime }\rightarrow E_{p}^{\prime }$ denote parallel translation along $%
\gamma $. We define 
\begin{equation*}
\widetilde{\phi }\left( p\right) =P_{\gamma }\phi \left( x_{p}\right) .
\end{equation*}%
Clearly, $\widetilde{\phi }\in \Gamma \left( B_{\varepsilon ,\overline{x}%
},E^{\prime }\right) $. Because the connection on $E^{\prime }$ is $G$%
-equivariant, parallel translation commutes with the action of $H_{\overline{%
x}}$, and $\widetilde{\phi }\in \Gamma \left( B_{\varepsilon ,\overline{x}%
},E^{\prime }\right) ^{H_{\overline{x}}}$. Abusing notation, we also denote
by $\widetilde{\phi }$ the corresponding $G$-invariant section of $%
T_{\varepsilon }^{\prime }$ (and thus of $W$) using the isomorphism in
Proposition \ref{tubeProposition}. This specific trivialization of $%
E^{\prime }$ produces test sections that can be used as in \cite{Shu1} to
obtain the upper bounds for the eigenvalues of $H_{s}^{G}$. We denote $\Phi
:\Gamma \left( \mathbb{R}^{m},E_{\overline{x}}^{\prime }\right) ^{H_{%
\overline{x}}}\rightarrow \Gamma \left( W,E^{\prime }\right) ^{G}$ to be the
trivialization $\phi \rightarrow \widetilde{\phi }$. We may extend $\Phi $
to act on the direct sum of the spaces $\Gamma \left( \mathbb{R}^{m},E_{%
\overline{x}}^{\prime }\right) ^{H_{\overline{x}}}$ (with fixed $\overline{x}
$ in each critical orbit $\mathcal{O}_{\overline{x}}$).

To obtain a lower bound on the eigenvalues of $H_{s}^{G}$ (or an upper bound
on the spectral counting function of $H_{s}^{G}$), we proceed exactly as in 
\cite{Shu1}. The functions in the partition of unity are chosen so that
those corresponding to neighborhoods of critical points are radial; then the
partition of unity will consist of invariant functions. Next, the IMS
localization formula allows us to localize to these small neighborhoods,
comparing the operators $\Phi ^{-1}H_{s}^{G}\Phi $ and $K^{G}$.
\end{proof}

\begin{corollary}
\label{LocalizationCorollary}If $D_{b}$ $\in \mathrm{Diff}_{b}^{\ast }\left(
M,E,\mathcal{F}\right) $ is a basic Dirac operator, and let $K^{G}$ be the
model operator constructed above, when the singular set of $Z$ is nonempty.
Then 
\begin{equation*}
\mathrm{ind}_{b}\left( D_{b}\right) =\dim \ker \left( \left( K^{G}\right)
^{+}\right) -\dim \ker \left( \left( K^{G}\right) ^{-}\right) .
\end{equation*}%
When the singular set of $Z$ is empty, $\mathrm{ind}_{b}\left( D_{b}\right)
=0$.
\end{corollary}

\begin{proof}
For each $s>0,$\ operators $H_{s}^{+}=s^{-1}$\ $\left( D_{s}^{\prime
}\right) ^{-}\left( D_{s}^{\prime }\right) ^{+}$\ and $H_{s}^{-}=$\ $%
s^{-1}\left( D_{s}^{\prime }\right) ^{+}\left( D_{s}^{\prime }\right) ^{-}$\
are positive transversally elliptic self-adjoint operators acting on
sections of vector bundles over the compact smooth manifold $W$.\ Therefore
the operators $H_{s}^{+}$\ and $H_{s}^{-}$\ have discrete spectra $\sigma
\left( H_{s}^{\pm }\right) \subset \left[ 0,+\infty \right) $ with finite
multiplicities. By Lemma \ref{discreteSpectrumLemma} and Theorem \ref%
{EquLocalizationThm copy(1)}, the spectra of $\left( K^{G}\right) ^{+}$ and $%
\left( K^{G}\right) ^{-}$ are also discrete and nonnegative.

Choose any real number $r>0,$\ so that $r\ $is strictly less than the least
positive number in the union of the spectra of $\left( K^{G}\right) ^{+}$\
and $\left( K^{G}\right) ^{-}.$\ Then for any $s>0$\ we have 
\begin{eqnarray*}
\mathrm{ind}_{b}\left( D_{b}\right) &=&\dim \ker \left( s^{-1}\left. \left(
D_{s}\right) ^{2}\right\vert _{\Gamma _{b}\left( M,E^{+},\mathcal{F}\right)
}\right) -\dim \ker \left( s^{-1}\left. \left( D_{s}\right) ^{2}\right\vert
_{\Gamma _{b}\left( M,E^{-},\mathcal{F}\right) }\right) , \\
&=&\dim \ker H_{s}^{+}-\dim \ker H_{s}^{-} \\
&=&\#\left\{ \sigma \left( H_{s}^{+}\right) \cap \left[ 0,r\right) \right\}
-\#\left\{ \sigma \left( H_{s}^{-}\right) \cap \left[ 0,r\right) \right\} ,
\end{eqnarray*}%
because $D_{s}^{+}$ is an isomorphism between the eigenspaces of $H_{s}^{+}$
and of $H_{s}^{-}$ corresponding to nonzero eigenvalues. By choosing $s$\
sufficiently large in the formula above and applying Theorem \ref%
{EquLocalizationThm copy(1)}, we obtain 
\begin{equation*}
\mathrm{ind}\left( D_{b}\right) =\dim \ker \left( \left( K^{G}\right)
^{+}\right) -\dim \ker \left( \left( K^{G}\right) ^{-}\right) .
\end{equation*}
\end{proof}

\begin{remark}
\label{IndexZeroRemark}With the notation of Section \ref{preliminaryperturb}%
, if $Z:=\left( Z^{+},\left( Z^{+}\right) ^{\ast }\right) \in \Gamma \left(
M,\mathrm{End}\left( E^{+}\oplus E^{-}\right) \right) $ is a smooth basic
bundle map that has no critical leaf closures and anticommutes with Clifford
multiplication by vectors orthogonal to leaf closures, then the corollary
implies that the index of the basic Dirac operator $D_{b}$ must be zero.
This is clear for several reasons, for instance 
\begin{equation*}
\ker \left( \left. \left( D_{s}\right) ^{2}\right\vert _{\Gamma _{b}\left(
M,E^{\pm },\mathcal{F}\right) }\right) =\ker \left( \left. \left(
D_{b}^{2}+s\left( ZD_{b}+D_{b}Z\right) +s^{2}Z^{2}\right) ^{2}\right\vert
_{\Gamma _{b}\left( M,E^{\pm },\mathcal{F}\right) }\right) .
\end{equation*}%
There exists $c>0$ such that for sufficiently large $s$, $s^{2}Z^{2}+s\left(
ZD_{b}+D_{b}Z\right) >cs^{2}\mathbf{1}$ , and the kernel is empty.
\end{remark}

\begin{corollary}
\label{OddCodimImpliesIndexZero}The index of a basic Dirac operator
corresponding to a basic Clifford bundle over a Riemannian foliation of odd
codimension is zero.
\end{corollary}

\begin{proof}
Proposition \ref{OddAlwaysExists} implies that there exists a basic
perturbation $Z$ that is everywhere invertible.
\end{proof}

\subsection{Clifford form of proper perturbations\label{SpecialFormSection}}

In this section, we show that in a neighborhood of a critical leaf closure,
the operator $Z$ may be continuously deformed so that it has a special form,
called \textbf{Clifford form}, near this leaf closure.

\begin{definition}
\label{CliffordFormDefinition}\vspace{0in}Suppose that $Z$ is a proper
perturbation of $D_{b}$. Then $Z$ is said to be \textbf{of Clifford form} if
near every critical leaf closure $\ell $, it has the form $\gamma \otimes
Z^{\dag }$ (or in odd codimension $\mathbf{1}\otimes \left( 
\begin{array}{cc}
0 & Z^{\dag } \\ 
-Z^{\dag } & 0%
\end{array}%
\right) $ ), and in coordinates $x_{k}$ in a disk normal to $\ell $, 
\begin{equation*}
Z^{\dag }=\sum x_{k}Z_{k}^{\dag }~,
\end{equation*}%
where 
\begin{equation*}
Z_{j}^{\dag }Z_{k}^{\dag }+Z_{k}^{\dag }Z_{j}^{\dag }=-2\left\langle
\partial _{j},\partial _{k}\right\rangle
\end{equation*}%
at the origin $\overline{x}\in \ell $.
\end{definition}

\begin{remark}
\label{LjRemark}If $Z$ is of Clifford form, then near $\ell $ each $Z_{j}$
anticommutes with each $c\left( \partial _{k}\right) $, and the operators $%
L_{j}=c\left( \partial _{j}\right) Z_{j}$ commute with each other at the
origin.
\end{remark}

In the following, we use the terminology \textbf{stable homotopy} of
endomorphisms and zeroth order operators in $K$-theory sense.

\begin{proposition}
\label{CliffordFormDeformationProposition}Let $Z$ be a proper perturbation
of $D_{b}$ . Then there exists stable homotopy $\Phi _{t}$ of proper
perturbations of $D_{b}$ such that

\begin{enumerate}
\item $\Phi _{0}=Z$ ,

\item $\Phi _{t}=Z$ outside a small neighborhood of the critical leaf
closure,

\item $\Phi _{1}$ has the same singular set as $Z$ , and

\item $\Phi _{1}$ is in Clifford form.
\end{enumerate}
\end{proposition}

\begin{proof}
We show that this homotopy can simultaneously be performed in the
neighborhood of each critical leaf closure $\ell $ of $Z$. Given a small
tubular neighborhood of $\ell $, let $U$ be the open neighborhood of the
origin in $\mathbb{R}^{m}$ identified with a ball orthogonal to $\ell $ at a
point. We are given an isometric action of the compact Lie group $H$ on $U$
and on $F=\left. E\right\vert _{U}$, corresponding to the holonomy of the
given leaf. We have $F\cong \mathbb{S}\otimes \mathcal{W}$, where the
Clifford action of $\mathbb{C}\mathrm{l}\left( \mathbb{R}^{m}\right) $ on
the bundle is of the form $c\otimes \mathbf{1}$, and where $\mathbb{S}$ the
irreducible spinor space for $\mathbb{C}\mathrm{l}\left( \mathbb{R}%
^{m}\right) $. Since $\mathcal{W}=\mathrm{Hom}_{\mathbb{C}\mathrm{l}\left( 
\mathbb{R}^{m}\right) }\left( \mathbb{S},F\right) $, the actions of $H$ on $%
\mathbb{S}$ and on $F$ induce the action on $\mathcal{W}$. The group $H$
commutes with the action of the chirality operator $\gamma \otimes \mathbf{1}
$ (for $\mathbb{C}\mathrm{l}\left( \mathbb{R}^{m}\right) $) and with the
equivariant bundle map $Z$. As in \cite[Propositions 2.7 and 2.10]%
{PrRiPerturb}, $Z$ has the form $\gamma \otimes Z^{\dag }$ or $\mathbf{1}%
\otimes \left( 
\begin{array}{cc}
0 & Z^{\dag } \\ 
-Z^{\dag } & 0%
\end{array}%
\right) $. Then $Z^{\dag }$ is equivariant with respect to the action on $%
\mathcal{W}$ (or $\mathcal{W}^{\dag }$ in the odd case with $\mathcal{W}=%
\mathcal{W}^{\dag }\oplus \mathcal{W}^{\dag }$). Since $H$ acts on $U$ by
isometries, then if $U$ is even-dimensional, the equivariant $K$-theory
satisfies 
\begin{equation*}
K_{H}\left( pt\right) \cong K_{H}\left( U\right) \cong R\left( H\right) 
\text{,}
\end{equation*}%
the representation ring of $H$, with generators given by $\left[ \mathbb{S}%
^{+}\otimes V_{\rho },\mathbb{S}^{-}\otimes V_{\rho },c\otimes \rho \right]
\in K_{H}\left( \mathbb{C}^{n/2}\right) $, where $c\left( x\right) $ is
Clifford multiplication by $x\in \mathbb{C}^{n/2}$ and $\rho :H\rightarrow
U\left( V_{\rho }\right) $ is an irreducible unitary representation. In more
generality (even in odd dimensions), $K_{H}\left( \mathbb{R}^{n}\right) $ is
generated by such triples; see \cite{Ech-Pf} for the specific details.
Therefore, we may stably homotope $Z^{\dag }$ in a neighborhood of the
origin to a Clifford multiplication-type operator, so that $Z^{\dag }$
satisfies the conditions in the definition above.

The homotopy is performed as follows for the even-dimensional case, and the
\linebreak odd-dimensional case is similar. We consider on a single
even-dimensional disk normal $U$ to the critical leaf closure. Restrict $%
Z^{\dag }$ to a sphere of radius $r$ in the normal disk, where $Z^{\dag }$
is nonsingular. It defines an element of $K_{H}\left( \mathbb{C}%
^{n/2}\right) $, and thus is in the same class as $\left[ \mathbb{S}%
^{+}\otimes V_{\rho },\mathbb{S}^{-}\otimes V_{\rho },c\otimes \rho \right] $
as above. After stabilizing, there exists an equivariant homotopy between
the two endomorphisms. Let $\ \widetilde{Z^{\dag }}\left( \tau \right) $
satisfy $\widetilde{Z^{\dag }}\left( r\right) =Z^{\dag }$, $\widetilde{%
Z^{\dag }}\left( \tau \right) =c\otimes \rho $ for $\tau \leq \frac{1}{2}r$,
and otherwise $\widetilde{Z^{\dag }}\left( \tau \right) $ is the $H$%
-equivariant homotopy between the two operators for $r\geq \tau \geq \frac{1%
}{2}r$. Then the homotopy between $Z$ and $\gamma \otimes \widetilde{Z^{\dag
}}$ is $\Phi _{t}=\left( 1-t\right) Z+t\left( \gamma \otimes \widetilde{%
Z^{\dag }}\right) $ near the critical leaf closure and is constant outside
the tubular neighborhood of radius $r$. This homotopy introduces no
additional critical leaf closures by construction.
\end{proof}

\subsection{Local calculations and the main theorem\label{MainTheoremSection}%
}

In this section we use Corollary \ref{LocalizationCorollary} to prove the
main theorem, Theorem \ref{Masterpiece}. Given any proper perturbation $Z$
of $D_{b}$, we stably deform it to a proper perturbation in Clifford form
using Proposition \ref{CliffordFormDeformationProposition}. We consider $%
H_{s}=\frac{1}{s}\left( D_{b}^{\prime }+sZ^{\prime }\right) ^{2}\in \mathrm{%
Diff}_{O\left( q\right) }^{\ast }\left( W,E^{\prime }\right) $, with the
local form of the model operator at $\overline{x}=0$ given by%
\begin{equation*}
K_{\overline{x}}=\left( \sum_{j}c\left( \partial _{j}\right) \partial
_{j}+\sum_{k}x_{k}Z_{k}\right) ^{2}:\Gamma \left( \mathbb{R}^{m},E_{%
\overline{x}}^{\prime }\right) \rightarrow \Gamma \left( \mathbb{R}^{m},E_{%
\overline{x}}^{\prime }\right) ,
\end{equation*}%
with each $Z_{k}$ a constant endomorphism that anticommutes with $c\left(
\partial _{j}\right) $ for each $j$, and such that $Z_{j}Z_{k}+Z_{k}Z_{j}=-2%
\delta _{jk}$, as described in the previous section.

As a consequence, the Hermitian operators $L_{j}=c\left( \partial
_{j}\right) Z_{j}$ commute with each other. Let $K_{\overline{x}}^{H_{%
\overline{x}}}$ be the restriction of $K_{\overline{x}}$ to $\Gamma \left( 
\mathbb{R}^{m},E_{\overline{x}}^{\prime }\right) ^{H_{\overline{x}}}$. Then 
\begin{eqnarray*}
K_{\overline{x}} &=&-\sum_{j=1}^{m}\partial _{j}^{2}+\sum_{j=1}^{m}c\left(
\partial _{j}\right) Z_{j}+\sum_{j=1}^{m}x_{j}^{2}Z_{j}^{2} \\
&=&\sum_{j=1}^{m}\left( -\partial _{j}^{2}+L_{j}+x_{j}^{2}L_{j}^{2}\right) .
\end{eqnarray*}%
The operators $L_{j}$ can be diagonalized simultaneously. Let $v$ be a
common eigenvector; let $\lambda _{j}$ be the eigenvalue of the operator $%
L_{j}$ corresponding to $v$. Letting $f$ be a scalar function of $x$, we
have 
\begin{equation*}
K_{\overline{x}}\left( fv\right) =\left( \sum_{j=1}^{m}\left( -\partial
_{j}^{2}+\lambda _{j}+\lambda _{j}^{2}x_{j}^{2}\right) f\right) v.
\end{equation*}%
The section $fv$ is in the kernel of $K_{\overline{x}}$ if and only if each $%
\lambda _{j}$ is negative, and, up to a constant, $fv=\exp \left( \frac{1}{2}%
\sum_{j}\lambda _{j}x_{j}^{2}\right) v$. The kernel $K_{\overline{x}}^{H_{%
\overline{x}}}$ is the $H_{\overline{x}}$-invariant subspace of $\ker K_{%
\overline{x}}$. The kernel of $K_{\overline{x}}$ is the intersection of the
direct sum of eigenspaces $E_{\lambda }\left( L_{j}\right) $ of $L_{j}$
corresponding to negative eigenvalues. Note that $L_{j}$ maps $E^{+}$ to
itself (call the restriction $L_{j}^{+}$), so that the dimension of $\ker K_{%
\overline{x}}^{H_{\overline{x}}}$ is simply the dimension of $%
\bigcap_{j}\left( \bigoplus_{\lambda <0}E_{\lambda }\left( L_{j}^{+}\right)
\right) ^{H_{\overline{x}}}$.

The calculation above and Corollary \ref{LocalizationCorollary} imply the
following theorem.

\begin{theorem}
\label{Masterpiece}Let $D_{b}$ $\in \mathrm{Diff}_{b}^{\ast }\left( M,E,%
\mathcal{F}\right) $ be a basic Dirac operator. Suppose that there exists a
proper perturbation (Definition \ref{properPerturbationDef}) $Z$ of $D_{b}$.
Then 
\begin{equation*}
\mathrm{ind}_{b}\left( D_{b}\right) =\sum_{\ell }\mathrm{ind}\left( Z,\ell
\right) ,
\end{equation*}%
where for each critical leaf closure $\ell $,%
\begin{equation*}
\mathrm{ind}\left( Z,\ell \right) =\left( \dim \left[ \bigcap_{j}\left(
\bigoplus_{\lambda <0}E_{\lambda }\left( L_{j}^{+}\left( \overline{x}\right)
\right) \right) ^{H_{\overline{x}}}\right] -\dim \left[ \bigcap_{j}\left(
\bigoplus_{\lambda <0}E_{\lambda }\left( L_{j}^{-}\left( \overline{x}\right)
\right) \right) ^{H_{\overline{x}}}\right] \right) ,
\end{equation*}%
where each critical leaf closures $\ell $ of $Z$ corresponds to a fixed
point $\overline{x}$ on the basic manifold $W$.
\end{theorem}

\section{Examples\label{ExamplesSection}}

The perturbations $Z$ in this section are already in Clifford form at the
critical leaf closures.

\begin{example}
\label{rotationExample}Consider the one dimensional foliation obtained by
suspending an irrational rotation on the standard unit sphere $S^{2}$. \ On $%
S^{2}$ we use the cylindrical coordinates $\left( \theta ,z\right) $,
related to the standard rectangular coordinates by $x^{\prime }=\sqrt{\left(
1-z^{2}\right) }\cos \theta $, $y^{\prime }=\sqrt{\left( 1-z^{2}\right) }%
\sin \theta $, $z^{\prime }=z$, $\theta \in \mathbb{R}\func{mod}2\pi $, $%
z\in \left[ -1,1\right] $. \ Let $\alpha $ be a fixed irrational multiple of 
$2\pi $, and let the three--manifold $M_{\alpha }=S^{2}\times \left[ 0,1%
\right] /\sim $, where $\left( \theta ,z,0\right) \sim \left( \theta +\alpha
,z,1\right) $. \ Endow $M_{\alpha }$ with the product metric on $T_{\left(
\theta ,z,t\right) }M_{\alpha }\cong T_{\left( \theta ,z\right) }S^{2}\times
T_{t}\mathbb{R}$. \ Let the foliation $\mathcal{F}_{\alpha }$ be defined by
the immersed submanifolds $\left\{ \left( \theta ^{\prime },z,\tau \right)
:\theta ^{\prime }=\theta +n\alpha ,~n\in \mathbb{Z},~\tau \in \mathbb{R}%
\func{mod}1\right\} $ through points $\left( z,\theta ,t\right) $. \ The
leaf closures for $|z|<1$ are two dimensional, and the closures
corresponding to the poles ($z=\pm 1$) are one dimensional. In the natural
metric, the foliation is Riemannian. We wish to calculate the basic Euler
characteristic of this foliation by using our theorem.\newline
Let $d_{b}:\Omega ^{\ast }\left( M_{\alpha },\mathcal{F}_{\alpha }\right)
\rightarrow \Omega ^{\ast }\left( M_{\alpha },\mathcal{F}_{\alpha }\right) $
be the exterior derivative restricted to basic forms, and let $\delta _{b}$
be the $L^{2}$-adjoint of $d_{b}$. The operator $D_{b}=d_{b}+\delta _{b}$
acting on even degree basic forms is called the basic de Rham operator. The
basic Laplacian is $\Delta _{b}=\left( d_{b}+\delta _{b}\right) ^{2}$, and
its kernel consists of basic harmonic forms. Since the Hodge theorem is
valid for Riemannian foliations, the standard argument shows that the index
of $D_{b}$ is the basic Euler characteristic $\chi \left( M_{\alpha },%
\mathcal{F}_{\alpha }\right) =\sum_{j\geq 0}\left( -1\right) ^{j}\dim \left(
H_{b}^{j}\left( M_{\alpha },\mathcal{F}_{\alpha }\right) \right) $. Here, $%
H_{b}^{j}\left( M_{\alpha },\mathcal{F}_{\alpha }\right) =\ker \left.
d_{b}\right\vert _{\Omega _{b}^{j}}\diagup \mathrm{Im}\left.
d_{b}\right\vert _{\Omega _{b}^{j-1}}$ is the basis cohomology group (see 
\cite{Mo}, \cite{Rein}).

For any bundle-like metric, the basic de Rham operator is 
\begin{equation*}
D_{b}=\left( d_{b}+\delta _{b}\right) =d_{b}+\delta _{T}+\kappa
_{b}\lrcorner ,
\end{equation*}
where $\kappa _{b}$ is the mean curvature one-form and $\alpha \lrcorner
=\left( \alpha \wedge \right) ^{\ast }$ for one-forms $\alpha $. Let the
perturbation be%
\begin{equation*}
Z=dz\wedge +dz\lrcorner ~.
\end{equation*}%
The reader may verify that $Z$ anticommutes with the principal symbol of $%
D_{b}$; as explained in Section \ref{admissiblePertSection}, this implies
that $D_{b}Z+ZD_{b}$ is zeroth order. The critical leaf closures correspond
to the poles $z=\pm 1$. In the coordinates $x,y$ near the poles, $%
Z=xZ_{1}+yZ_{2}$, where 
\begin{equation*}
Z_{1}\left( \pm 1\right) =\mp \left( dx\wedge +dx\lrcorner \right)
,~Z_{2}\left( \pm 1\right) =\mp \left( dy\wedge +dy\lrcorner \right) .
\end{equation*}%
Then%
\begin{eqnarray*}
L_{j} &=&c\left( \partial _{j}\right) Z_{j}\left( \pm 1\right) \\
L_{1} &=&\mp \left( dx\wedge -dx\lrcorner \right) \left( dx\wedge
+dx\lrcorner \right) =\mp \left( dx\wedge dx\lrcorner -dx\lrcorner dx\wedge
\right) \\
L_{2} &=&\mp \left( dy\wedge -dy\lrcorner \right) \left( dy\wedge
+dy\lrcorner \right) =\mp \left( dy\wedge dy\lrcorner -dy\lrcorner dy\wedge
\right)
\end{eqnarray*}

At the north pole $z=+1$,%
\begin{gather*}
E_{-1}\left( L_{1}\right) =\mathrm{span}\left\{ dx,dx\wedge dy\right\}
,E_{-1}\left( L_{2}\right) =\mathrm{span}\left\{ dy,dx\wedge dy\right\}  \\
E_{-1}\left( L_{1}\right) \cap E_{-1}\left( L_{2}\right) =\mathrm{span}%
\left\{ dx\wedge dy\right\} .
\end{gather*}%
Similarly, at the south pole ( $z=-1$ ),%
\begin{equation*}
E_{-1}\left( L_{1}\right) \cap E_{-1}\left( L_{2}\right) =\mathrm{span}%
\left\{ 1\right\} 
\end{equation*}%
Here $E^{+}=\Lambda ^{\text{even}}T_{0}^{\ast }\mathbb{R}^{2}$, $%
E^{-}=\Lambda ^{\text{odd}}T_{0}^{\ast }\mathbb{R}^{2}$, and the vector
subspaces found are $O\left( 2\right) $-invariant, so that 
\begin{eqnarray*}
\mathrm{ind}\left( Z,z=\pm 1\right)  &=&\left( \dim \left[ \bigcap_{j}\left(
\bigoplus_{\lambda <0}E_{\lambda }\left( L_{j}^{+}\left( \overline{x}\right)
\right) \right) ^{H_{\overline{x}}}\right] -\dim \left[ \bigcap_{j}\left(
\bigoplus_{\lambda <0}E_{\lambda }\left( L_{j}^{-}\left( \overline{x}\right)
\right) \right) ^{H_{\overline{x}}}\right] \right)  \\
&=&1-0=1. \\
\chi \left( M_{\alpha },\mathcal{F}_{\alpha }\right)  &=&\mathrm{ind}\left(
D_{b}\right) =1+1=2.
\end{eqnarray*}

We now directly calculate the Euler characteristic of this foliation. Since
the foliation is taut, the standard Poincare duality works \cite{KT3} \cite%
{KTduality} , and $H_{b}^{0}\left( M\right) \cong H_{b}^{2}\left( M\right)
\cong \mathbb{R}$ . It suffices to check the dimension $h^{1}$ of the
cohomology group $H_{b}^{1}\left( M\right) $. \ Then the basic Euler
characteristic is $\chi \left( M_{\alpha },\mathcal{F}_{\alpha }\right)
=1-h^{1}+1=2-h^{1}$. It was shown in \cite[Example 10.4]{BKR3} that $h_{1}=0$%
, so indeed $\chi \left( M_{\alpha },\mathcal{F}_{\alpha }\right) =2$.
\end{example}

\begin{example}
We will compute the basic Euler characteristic of the Carri\`{e}re example
from \cite{Car} in the $3$-dimensional case. Let $A$ be a matrix in $\mathrm{%
SL}_{2}(\mathbb{Z})$ with $\mathrm{tr}\left( A\right) >2$. We denote
respectively by $V_{1}$ and $V_{2}$ the eigenvectors associated with the
eigenvalues $\lambda $ and $\frac{1}{\lambda }$ of $A$ with $\lambda >1$
irrational. Let the hyperbolic torus $\mathbb{T}_{A}^{3}$ be the quotient of 
$\mathbb{T}^{2}\times \mathbb{R}$ by the equivalence relation which
identifies $(m,t)$ to $(A(m),t+1)$. The flow generated by the vector field $%
V_{1}$ can be made into a Riemannian foliation with a bundle-like metric $g$%
, as follows. Let $\left( x,y,t\right) $ denote the local coordinates in the 
$V_{1}$ , $V_{2}$ , and $\mathbb{R}$ directions, respectively, and let 
\begin{equation*}
g=\lambda ^{2t}dx^{2}+\lambda ^{-2t}dy^{2}+dt^{2}.
\end{equation*}

Next, let 
\begin{equation*}
Z=\cos \left( 2\pi t\right) dt\wedge +\cos \left( 2\pi t\right) dt\lrcorner .
\end{equation*}%
The operator $Z$ is a proper perturbation, as in the previous example. The
critical leaf closures are those points of $\mathbb{T}_{A}^{3}$
corresponding to $t=\frac{1}{4}$, $\frac{3}{4}$. The bundle map has the
local form $Z\left( t\right) =-2\pi \left( t-\frac{1}{4}\right) \left(
dt\wedge +dt\lrcorner \right) $ near $t=\frac{1}{4}$ and $Z\left( t\right)
=2\pi \left( t-\frac{3}{4}\right) \left( dt\wedge +dt\lrcorner \right) $
near $t=\frac{3}{4}$ on the one-dimensional orthogonal disk, and so that%
\begin{equation*}
Z_{1}=\left\{ 
\begin{array}{cc}
-2\pi \left( dt\wedge +dt\lrcorner \right)  & \text{if }t=\frac{1}{4} \\ 
2\pi \left( dt\wedge +dt\lrcorner \right)  & \text{if }t=\frac{3}{4}%
\end{array}%
\right. ,
\end{equation*}%
and thus%
\begin{eqnarray*}
L_{1} &=&c\left( \partial _{t}\right) Z_{1} \\
&=&\left\{ 
\begin{array}{cc}
-2\pi \left( dt\wedge -dt\lrcorner \right) \left( dt\wedge +dt\lrcorner
\right)  & \text{if }t=\frac{1}{4} \\ 
2\pi \left( dt\wedge -dt\lrcorner \right) \left( dt\wedge +dt\lrcorner
\right)  & \text{if }t=\frac{3}{4}%
\end{array}%
\right.  \\
&=&\left\{ 
\begin{array}{cc}
-2\pi \left( dt\wedge dt\lrcorner -dt\lrcorner dt\wedge \right)  & \text{if }%
t=\frac{1}{4} \\ 
2\pi \left( dt\wedge dt\lrcorner -dt\lrcorner dt\wedge \right)  & \text{if }%
t=\frac{3}{4}%
\end{array}%
\right. 
\end{eqnarray*}%
There is no holonomy, so in both cases $H_{\overline{x}}$ is trivial. We
compute the eigenvalues of $L_{1}$ on $E^{+}=\mathrm{span}\left\{
1,dy_{p}\wedge dt_{p}\right\} $ and on $E^{-}=\mathrm{span}\left\{
dy_{p},dt_{p}\right\} $: $\left( dt\wedge dt\lrcorner -dt\lrcorner dt\wedge
\right) $ has eigenvalues $-1,1$ on $E^{+}$, $-1,1$ on $E^{-}$. Thus, in the
index formula, 
\begin{eqnarray*}
\mathrm{ind}\left( Z,t=\frac{1}{4}~\text{or }\frac{3}{4}\right)  &=&\left(
\dim \left[ E_{-1}\left( L_{1}^{+}\right) \right] -\dim \left[ E_{-1}\left(
L_{1}^{-}\right) \right] \right)  \\
&=&1-1=0,
\end{eqnarray*}%
so that 
\begin{equation*}
\chi \left( \mathbb{T}_{A}^{3},\mathcal{F}\right) =\mathrm{ind}\left(
D_{b}\right) =0+0=0.
\end{equation*}%
It was shown in (\cite[Example 10.8]{BKR3}) that 
\begin{equation*}
H_{b}^{0}\left( \mathbb{T}_{A}^{3},\mathcal{F}\right) \cong H_{b}^{1}\left( 
\mathbb{T}_{A}^{3},\mathcal{F}\right) \cong \mathbb{R},H_{b}^{2}\left( 
\mathbb{T}_{A}^{3},\mathcal{F}\right) \cong 0,
\end{equation*}%
which verifies our computation.
\end{example}

\begin{example}
(Localization when $D_{b}Z+ZD_{b}$ is a first order operator.\label%
{signatureExample}) We will examine a transverse signature operator for a
one-dimensional foliation. The manifold is the suspension of a $\mathbb{Z}$%
-action on $\mathbb{CP}^{2}$. Specifically, let $\left( \theta _{0},\theta
_{1},\theta _{2}\right) \in T^{3}=S^{1}\times S^{1}\times S^{1}$ be a
generator of a discrete dense subgroup of $T^{3}$, and let $\phi \left( %
\left[ z_{0},z_{1},z_{2}\right] \right) =\left( \theta _{0},\theta
_{1},\theta _{2}\right) \cdot \left[ z_{0},z_{1},z_{2}\right] =\left[ \exp
\left( i\theta _{0}\right) z_{0},\exp \left( i\theta _{1}\right) z_{1},\exp
\left( i\theta _{2}\right) z_{2}\right] $ for all $\left[ z_{0},z_{1},z_{2}%
\right] \in \mathbb{CP}^{2}$. The suspension is \linebreak $M=\left( \mathbb{%
CP}^{2}\times \mathbb{R}\right) \diagup \sim $, where $\left( z,t\right)
\sim \left( \phi \left( z\right) ,t+1\right) $ for all $z\in \mathbb{CP}^{2}$%
, $t\in \mathbb{R}$. The $t$-parameter curves form a Riemannian foliation of 
$M$ with a natural bundle-like metric. Next, let $D_{b}$ be the transverse
signature operator, the pullback of the signature operator on $\mathbb{CP}%
^{2}$ via the local projections $\pi _{1}:U\rightarrow \mathbb{CP}^{2}$ for $%
U$ open in $M$.We may identify the bundle $Q$ with $T\mathbb{CP}^{2}$. Then $%
c:Q^{\ast }\rightarrow \mathrm{End}\left( \Lambda ^{\ast }Q^{\ast }\right) $
is the standard Clifford action by cotangent vectors, so that the basic
Dirac operator on $M$ is the pullback the operator $D_{b}=d+d^{\ast }$on $%
T^{3}$-invariant forms on $\mathbb{CP}^{2}$. The grading on the bundle $%
E=\Lambda ^{\ast }Q^{\ast }$ is given by the standard signature involution
operator on forms on $\mathbb{CP}^{2}$. Next, let $Z=ic\left( V^{\ast
}\right) $, where $V$ is the infinitesimal generator of the $S^{1}$action $%
t\mapsto \left[ \exp \left( it\alpha _{0}\right) z_{0},\exp \left( it\alpha
_{1}\right) z_{1},\exp \left( it\alpha _{2}\right) z_{2}\right] $ on $%
\mathbb{CP}^{2}$; we assume that $\alpha _{0},\alpha _{1},\alpha _{2}$ are
linearly independent over $\mathbb{Q}$. The reader may check that $%
D_{b}Z+ZD_{b}$ is a first order differential operator but is bounded on the
basic forms.

The critical leaf closures are the circles corresponding to the points $%
\left[ 1,0,0\right] $, $\left[ 0,1,0\right] $, $\left[ 0,0,1\right] $ of $%
\mathbb{CP}^{2}$. In the coordinates $\left[ 1,z_{1},z_{2}\right] $ near $%
\overline{x}=\left[ 1,0,0\right] $, $V=$ $Z=ic\left( V^{\ast }\right)
=x_{1}Z_{1}+y_{1}Z_{2}+x_{2}Z_{3}+y_{2}Z_{4}$, where 
\begin{eqnarray*}
Z_{1}\left( \overline{x}\right)  &=&i\left( \alpha _{1}-\alpha _{0}\right)
\left( dy_{1}\wedge -dy_{1}\lrcorner \right) =i\left( \alpha _{1}-\alpha
_{0}\right) c\left( dy_{1}\right)  \\
Z_{2}\left( \overline{x}\right)  &=&-i\left( \alpha _{1}-\alpha _{0}\right)
c\left( dx_{1}\right)  \\
Z_{3}\left( \overline{x}\right)  &=&i\left( \alpha _{2}-\alpha _{0}\right)
c\left( dy_{2}\right)  \\
Z_{4}\left( \overline{x}\right)  &=&-i\left( \alpha _{2}-\alpha _{0}\right)
c\left( dx_{2}\right) 
\end{eqnarray*}%
Then%
\begin{eqnarray*}
L_{1} &=&c\left( dx_{1}\right) Z_{1}\left( \overline{x}\right)  \\
&=&i\left( \alpha _{1}-\alpha _{0}\right) c\left( dx_{1}\right) c\left(
dy_{1}\right) =L_{2} \\
L_{3} &=&L_{4}=i\left( \alpha _{2}-\alpha _{0}\right) c\left( dx_{2}\right)
c\left( dy_{2}\right) .
\end{eqnarray*}%
An easy computation shows that invariant self-dual forms $E^{+}$ and
invariant anti-self-dual forms $E^{-}$ at $\overline{x}$ are%
\begin{equation*}
E^{\pm }=\mathrm{span}\left\{ 
\begin{array}{c}
dx_{1}\wedge dy_{1}\wedge dx_{2}\wedge dy_{2}\mp 1, \\ 
dx_{1}\wedge dy_{1}\pm dx_{2}\wedge dy_{2}%
\end{array}%
\right\} 
\end{equation*}%
We see that 
\begin{eqnarray*}
ic\left( dx_{q}\right) c\left( dy_{q}\right) \left( 1\pm idx_{q}\wedge
dy_{q}\right)  &=&\pm \left( 1\pm idx_{q}\wedge dy_{q}\right)  \\
L_{1}\left( 1\pm idx_{1}\wedge dy_{1}\right)  &=&L_{2}\left( 1\pm
idx_{1}\wedge dy_{1}\right) =\pm \left( \alpha _{1}-\alpha _{0}\right)
\left( 1\pm idx_{1}\wedge dy_{1}\right)  \\
L_{3}\left( 1\pm idx_{2}\wedge dy_{2}\right)  &=&L_{4}\left( 1\pm
idx_{2}\wedge dy_{2}\right) =\pm \left( \alpha _{2}-\alpha _{0}\right)
\left( 1\pm idx_{2}\wedge dy_{2}\right) 
\end{eqnarray*}%
Thus, 
\begin{equation*}
\bigcap_{j}\left( \bigoplus_{\lambda <0}E_{\lambda }\left( L_{j}\left( 
\overline{x}\right) \right) \right) =\mathrm{span}\left\{ \left( 1-\mathrm{%
sgn}\left( \alpha _{1}-\alpha _{0}\right) idx_{1}\wedge dy_{1}\right) \wedge
\left( 1-\mathrm{sgn}\left( \alpha _{2}-\alpha _{0}\right) idx_{2}\wedge
dy_{2}\right) \right\} .
\end{equation*}%
If $\mathrm{sgn}\left( \alpha _{1}-\alpha _{0}\right) =-\mathrm{sgn}\left(
\alpha _{2}-\alpha _{0}\right) $, then\newline
$\displaystyle\bigcap_{j}\left( \bigoplus_{\lambda <0}E_{\lambda }^{+}\left(
L_{j}\left( \overline{x}\right) \right) \right) =\left\{ 0\right\} ,$%
\begin{multline*}
\bigcap_{j}\left( \bigoplus_{\lambda <0}E_{\lambda }^{-}\left( L_{j}\left( 
\overline{x}\right) \right) \right)  \\
=\mathrm{span}\left\{ 
\begin{array}{c}
\left( 1-\mathrm{sgn}\left( \alpha _{1}-\alpha _{0}\right) idx_{1}\wedge
dy_{1}\right) \wedge \left( 1-\mathrm{sgn}\left( \alpha _{2}-\alpha
_{0}\right) idx_{2}\wedge dy_{2}\right) 
\end{array}%
\right\} .
\end{multline*}%
If $\mathrm{sgn}\left( \alpha _{1}-\alpha _{0}\right) =\mathrm{sgn}\left(
\alpha _{2}-\alpha _{0}\right) $, then%
\begin{multline*}
\bigcap_{j}\left( \bigoplus_{\lambda <0}E_{\lambda }^{+}\left( L_{j}\left( 
\overline{x}\right) \right) \right)  \\
=\mathrm{span}\left\{ 
\begin{array}{c}
\left( 1-\mathrm{sgn}\left( \alpha _{1}-\alpha _{0}\right) idx_{1}\wedge
dy_{1}\right) \wedge \left( 1-\mathrm{sgn}\left( \alpha _{2}-\alpha
_{0}\right) idx_{2}\wedge dy_{2}\right) 
\end{array}%
\right\} ,
\end{multline*}%
$\displaystyle\bigcap_{j}\left( \bigoplus_{\lambda <0}E_{\lambda }^{-}\left(
L_{j}\left( \overline{x}\right) \right) \right) =\left\{ 0\right\} .$

Therefore, by Theorem \ref{Masterpiece}, 
\begin{eqnarray*}
\mathrm{ind}\left( Z,\left[ 1,0,0\right] \right)  &=&\left( \dim \left[
\bigcap_{j}\left( \bigoplus_{\lambda <0}E_{\lambda }\left( L_{j}^{+}\left( 
\overline{x}\right) \right) \right) ^{H_{\overline{x}}}\right] -\dim \left[
\bigcap_{j}\left( \bigoplus_{\lambda <0}E_{\lambda }\left( L_{j}^{-}\left( 
\overline{x}\right) \right) \right) ^{H_{\overline{x}}}\right] \right)  \\
&=&\left\{ 
\begin{array}{ll}
1 & \text{if }\mathrm{sgn}\left( \alpha _{1}-\alpha _{0}\right) =\mathrm{sgn}%
\left( \alpha _{2}-\alpha _{0}\right) \text{ } \\ 
-1~~ & \text{if }\mathrm{sgn}\left( \alpha _{1}-\alpha _{0}\right) =-\mathrm{%
sgn}\left( \alpha _{2}-\alpha _{0}\right) 
\end{array}%
\right. .
\end{eqnarray*}%
Similarly, 
\begin{eqnarray*}
\mathrm{ind}\left( Z,\left[ 0,1,0\right] \right)  &=&\left\{ 
\begin{array}{ll}
1 & \text{if }\mathrm{sgn}\left( \alpha _{0}-\alpha _{1}\right) =\mathrm{sgn}%
\left( \alpha _{2}-\alpha _{1}\right) \text{ } \\ 
-1~~ & \text{if }\mathrm{sgn}\left( \alpha _{0}-\alpha _{1}\right) =-\mathrm{%
sgn}\left( \alpha _{2}-\alpha _{1}\right) 
\end{array}%
\right. , \\
\mathrm{ind}\left( Z,\left[ 0,0,1\right] \right)  &=&\left\{ 
\begin{array}{ll}
1 & \text{if }\mathrm{sgn}\left( \alpha _{0}-\alpha _{2}\right) =\mathrm{sgn}%
\left( \alpha _{1}-\alpha _{2}\right) \text{ } \\ 
-1~~ & \text{if }\mathrm{sgn}\left( \alpha _{0}-\alpha _{2}\right) =-\mathrm{%
sgn}\left( \alpha _{1}-\alpha _{2}\right) 
\end{array}%
\right. .
\end{eqnarray*}%
Thus,%
\begin{equation*}
\mathrm{ind}_{b}\left( D_{b}\right) =\sum_{\ell }\mathrm{ind}\left( Z,\ell
\right) =1+1-1=1
\end{equation*}%
in all cases. This calculation agrees with the fact that the signature of $%
\mathbb{CP}^{2}$ is 1, and the space of degree-two harmonic forms has
dimension one and is also invariant under the isometry group of $\mathbb{CP}%
^{2}$.
\end{example}

\end{document}